\documentclass[11pt]{elsarticle}
 
\usepackage{amsthm, amsfonts, amsxtra, amssymb, amscd}

\usepackage{amsfonts,amsmath,amssymb,amscd,amsthm}
\usepackage[english]{babel} 
\usepackage[all]{xy}
\usepackage[backref, colorlinks, linktocpage, citecolor = blue, linkcolor = blue]{hyperref}

 \newtheorem{lemma}[subsubsection]{Lemma}
 \newtheorem{theorem}[subsubsection]{Theorem}
 \newtheorem{proposition}[subsubsection]{Proposition}
 \newtheorem{corollary}[subsubsection]{Corollary}

\theoremstyle{definition}
\newtheorem*{definition*}{Definition}

\newtheorem{definition}[subsubsection]{Definition}
\newtheorem*{example*}{Example}
\newtheorem{example}[subsubsection]{Example}

\theoremstyle{remark}
\newtheorem*{remark*}{Remark}
\newtheorem{remark}[subsubsection]{Remark}

\newcommand{\pd}[2]{\dfrac{\partial#1}{\partial#2}}

\renewcommand{\phi}{\varphi}

\newcommand{\be}{\begin{enumerate}}
\newcommand{\ee}{\end{enumerate}}

\title{$T$--ideals of Cayley Hamilton algebras}
\author{Claudio Procesi}
\begin{document}\begin{abstract} We develop the  general Theory of    Cayley Hamilton algebras  and we compare this with the theory of pseudocharacters.  We finally characterize prime  $T$--ideals for Cayley Hamilton  algebras and discuss some of their geometry.
\end{abstract}

\maketitle
\hfill{\it\small To the memory of  T. A. Springer}
\tableofcontents
\section{Foreword}  A basic fact for an $n\times n$   matrix $a$ (with entries in a commutative ring)  is the construction of its characteristic polynomial $\chi_a(t):=\det(t-a)$ and the Cayley Hamilton theorem $\chi_a(a)=0$.

The notion  of  Cayley Hamilton algebra (CH algebras for short), see Definition \ref{CHalg},  was introduced in 1987 by Procesi \cite{P5} as an axiomatic treatment  of the Cayley Hamilton theorem. This was done in order to clarify the Theory of $n$--dimensional representations, cf. Definition \ref{ndi}, of an associative and in general noncommutative algebra $R$ (from now on just called {\em algebra}).  In particular the aim was to describe a {\em very strong}  embedding Theorem into matrices over a commutative ring (Theorem \ref{immu}).\smallskip

The theory was developed only  in characteristic 0, for two reasons.  The first being that, at that time, it was not clear to the author if the characteristic free results of Donkin \cite{Don} and Zubkov \cite{Zub}  were sufficient to found the theory in general. The second reason was mostly because it looked  not likely that the {\em main theorem} \ref{immu}  could possibly hold in general.

The first concern can now be considered to have a positive solution, due to the contributions of several people, and we may take the book \cite{depr0} as reference. As for the second, that is the main theorem in positive characteristic, the issue remains unsettled. The present author feels that it should not be true in general but has no counterexamples.\medskip

Independently, studying deformations of representations of   Galois groups, the theory of {\em pseudocharacters} or  {\em pseudorepresentations} was developed by several authors see \cite{Wi}, \cite{Tay}. We shall discuss the relationships between the two approaches in Theorem \ref{Gimmu}.\smallskip

A  partial theory in general characteristics replacing the trace with the determinant, a {\em norm}, appears already in Procesi \cite{P2} and  \cite{P6}.

  For the general definition, see  Chenevier \cite{Che}.
 
The original definition over $\mathbb Q$  is through the axiomatization of a {\em trace}, see \S \ref{axtr},   and closer to the Theory of {\em pseudocharacters} Definition \ref{pse}.  The definition via trace is also closer to the language of Universal algebra, while the one using norms is more categorical in nature. In this paper we restrict to characteristic 0 and traces, a discussion of the general case will appear elsewhere \cite{ppr}.

The paper is divided in three parts. In the first we recall some general facts which are treated in detail in the recent book with Aljadeff, Giambruno and Regev \cite{agpr} to which we refer for the proofs. 

In the second part we develop in a systematic way the Theory of Cayley--Hamilton algebras. We define the notion of prime trace algebras and show that the corresponding spectrum is a functor for Cayley--Hamilton algebras Theorem \ref{iso}. Finally, in the third part, from a clear structure theorem for prime and simple trace algebras, Theorem \ref{pos} and Theorem \ref{ssre1},  we give a classification of prime $T$--ideals in the free Cayley--Hamilton algebra, Theorem \ref{prst0},. We then study the structure of the associated relatively free algebras both from a geometric Theorems \ref{cod}, \ref{norm} and an algebraic point of view, Theorem \ref{ultt}.  

 A few    results of the second part already appear with restrictive hypotheses in the literature,  in general this material is new. The third part is new.
\vfill\eject

\part{Background}
\section{Invariants and representations}
\subsection{$n$--dimensional representations} For a given $n\in\mathbb N$ and a ring $A$  by $M_n(A)$ we denote the ring of $n\times n$ matrices with coefficients in $A$. By a symbol $(a_{i,j})$  we denote a matrix with entries $ a_{i,j}\in A,\ i,j=1,\ldots,n    $.

In particular we will usually assume $A$  commutative so that the construction $A\mapsto M_n(A)$ is a functor from the category $\mathcal C$ of commutative rings to that $\mathcal R$ of associative rings. To a map $f:A\to B$ is associated a  map $M_n(f):M_n(A)\to M_n(B)$  in the obvious way $M_n(f)((a_{i,j})):= (f(a_{i,j}))$.
\begin{definition}\label{ndi}
By an  {\em $n$--dimensional representation}  of a ring $R$ we mean a homomorphism $f:R\to M_n(A)$ with $A$ commutative.
\end{definition}
The set valued functor $A\mapsto \hom_{\mathcal R}(R, M_n(A))$ is representable. That is:
\begin{proposition}\label{laTR}
There is a commutative ring $T_n(R)$  and a natural isomorphism $\mathrm j_A $
$$ \hom_{\mathcal R}(R, M_n(A))\stackrel{\mathrm j_A }\simeq  \hom_{\mathcal C}(T_n(R),  A  ),\quad \mathrm j_A :f\mapsto \bar f$$ given by the commutative diagram $f=M_n(\bar f)\circ \mathtt j_R$:
\begin{equation}\label{unpp}
\xymatrix{  R\ar@{->}[r]^{\mathtt j_R\qquad}\ar@{->}[rd]_f&M_n(T_n(R))\ar@{->}[d]^{M_n(\bar f)}\\
&M_n(A) } .
\end{equation} 
\end{proposition} The map $\mathrm j_R:R\to M_n(T_n(R))$ is called the {\em universal $n$--dimensional representation of $R$ } or {  \em the universal map into $n\times n$ matrices}.

Of course it is possible that $R$ has no   $n$--dimensional representations, in which case $T_n(R)=\{0\}$. \medskip

\begin{remark}\label{rinal}
The same discussion can be performed when $R$  is in the category $\mathcal R_F$ of algebras over a commutative ring $F.$ 

Then the functor   $A\mapsto \hom_{\mathcal R_F}(R, M_n(A))$ is on commutative $F$ algebras and $T_n(R)$ is an $F$ algebra.
\end{remark}
From now on $F$ will be a fixed commutative ring.\smallskip

The construction of $\mathrm j_R$ is in two steps.
\subsubsection{Step 1, the free algebra} One easily sees that when $R=F\langle x_i\rangle_{i\in \mathcal I}$ is a free algebra  then:
\begin{proposition}\label{genma}
$T_n(R)=F[\xi_{h,k}^{(i)}]$  is the polynomial algebra over $F$ in the variables $\xi_{h,k}^{(i)},\ i\in \mathcal I,\ h,k=1,\ldots, n$ and $\mathrm j_R(x_i)=\xi_i:=(\xi_{h,k}^{(i)})$ is the {\em generic $n\times n$ matrix} with entries $\xi_{h,k}^{(i)}$.
\end{proposition} 

\begin{definition}\label{genmad}
The subalgebra  $F_n(\mathcal I):=F\langle \xi_i\rangle$  of  $M_n( F[\xi_{h,k}^{(i)}]),$ $ i\in \mathcal I,\ h,k=1,\ldots, n  $ generated by the matrices $\xi_i$  is called the algebra of  {\em generic matrices.}
\end{definition}

If $\mathcal I$ has $\ell$ elements we also denote $F_n (\mathcal I)= F_n (\ell) $. 
 A classical Theorem of Amitsur  states that: \begin{theorem}\label{teoAmi}
If $F$ is a domain, then  $F_n(\mathcal I) $ is a domain. 
If $\ell\geq  2$, then   $F_n (\ell) $ has a division ring of quotients $D_n(\ell)$  which is of dimension $n^2$ over its center $Z_n(\ell)$. 
\end{theorem}These algebras have been extensively studied.    Together with these algebras one has:

\begin{definition}\label{genmadt}\strut  \begin{enumerate}\item The commutative algebra  $\mathcal C_n(\ell)\subset  Z_n(\ell)$ generated, for all $a\in    F_n (\ell) $, by the coefficients $\sigma_i(a)$ of the characteristic polynomial  $$\det(t-a)=t^n+\sum_{i=1}^n(-1)^i\sigma_i(a)t^{n-i},\ \forall a\in F_n(\ell).$$   
\item 
Next define $\mathcal S_n(\ell)=F_n(\ell)\mathcal C_n(\ell)\subset D_n(\ell),$   (called the {\em trace algebra}).  \end{enumerate}

  \end{definition}
  One can understand $\mathcal S_n(\ell)$ and $\mathcal C_n(\ell)$  by invariant theory, see Theorem \ref{FFT} and  by Universal algebra, see Theorem \ref{SFT}.\label{prSl}

The invariant theory involved is presented in \cite{P3} when $F$ is a field of characteristic 0, and may be considered as the {\em first fundamental theorem} of matrix invariants. For a characteristic free treatment the Theorem is due to Donkin \cite{Don}. In general assume that $F$ is an infinite field:
\begin{theorem}\label{FFT}
The algebra $\mathcal C_n(\ell)$ is the algebra of polynomial invariants  under the simultaneous action of $GL(n,F)$ by conjugation on the space $M_n(F)^\ell$ of $\ell$--tuples of $n\times n$ matrices.\smallskip

The algebra $\mathcal S_n(\ell)$ is the algebra of $GL(n,F)$--equivariant polynomial maps  from  the space $M_n(F)^\ell$ of $\ell$--tuples of $n\times n$ matrices to $M_nF)$.

\end{theorem}
As usual together with a   first fundamental theorem one may ask for a  {\em second fundamental theorem} which was proved independently by Procesi \cite{P3} and Razmyslov \cite{R}  when $F$ has characteristic 0 and by Zubkov \cite{Zub} in general, see the book \cite{depr0}. 

The best way to explain this Theorem in characteristic 0, and which is the basis of the present work, is to set it into the language of universal algebra by introducing the category of algebras with trace and trace identities, see \S \ref{axtr}.
\smallskip

\subsubsection{Step 2, all  algebras}

A general algebra $R$  can be presented as a quotient $R=F\langle x_i\rangle/I$ of a free algebra. Then $j_{F\langle x_i\rangle}(I)$ generates in $M_n( F[\xi_{h,k}^{(i)}]),\ i\in \mathcal I,\ h,k=1,\ldots, n  $ an ideal which is, as any ideal in a matrix algebra, of the form  $M_n(J),$ with $J$ an ideal of $F[\xi_{h,k}^{(i)}]$. Then the universal map for the algebra $R$ is given by  $\mathrm j_R:R\to M_n(F[\xi_{h,k}^{(i)}]/J).$  By the universal property this is   independent of the presentation of $R$.

Again one may add to $R$  the algebra $\mathcal  T_n(R) $ generated    by the coefficients of the characteristic polynomial  $\sigma_i(a),\ \forall a\in \mathrm j_R(R)$.  

\section{Symmetry\label{symme}} \subsection{The projective group} The functor $A\mapsto \hom_{\mathcal R}(R, M_n(A))$  has a group of symmetries: the {\em projective linear group $PGL(n)$}.

It is best to define this as a representable group valued functor  on the category $\mathcal C$ of commutative rings. The functor associates to a commutative ring $A$ the group $\mathfrak G_n(A):=Aut_A(M_n(A))$ of $A$--linear automorphisms of the matrix algebra $M_n(A)$.  To a morphism $f:A\to B$  one has an  associated morphism $f_*:\mathfrak G_n(A)\to \mathfrak G_n(B)$  and  the  commutative diagram:
\begin{equation}\label{unpp0}
\xymatrix{  M_n(A)\ar@{->}[r]^{g}\ar@{->}[d]_{M_n(f)}&M_n(A)\ar@{->}[d]^{M_n(f)}\\
M_n(B)\ar@{->}[r]^{f_*(g)} &M_n(B) } .
\end{equation}
One has a natural homomorphism of the general linear group $GL(n,A)$ to $\mathfrak G_n(A)$ which associates to an invertible matrix $X$  the inner automorphism  $a\mapsto XaX^{-1}$.

The functor {\em general linear group $A\mapsto GL(n,A)$} is represented by   the  Hopf algebra   $\mathbb Z[x_{i,j}][d^{-1}],\ i,j=1,\ldots,n$   with $d=\det(X),\ X:=(x_{i,j})$  with the usual  structure given compactly by comultiplication $\delta$, antipode $S$ and counit $\epsilon$:
$$\delta(X)=X\otimes X,\ S(X):=X^{-1},\ \epsilon:X\to 1_n.$$

The functor $A\mapsto \mathfrak G_n(A)$ is represented by  the sub  Hopf algebra  of $GL(n,A)$, $P_n\subset \mathbb Z[x_{i,j}][d^{-1}]$  formed  of elements homogeneous of degree 0.

 It has a basis, over $\mathbb Z$, of elements $ad^{-h}$  where $a$  is a doubly standard tableaux with no rows of length $n$  and of degree $h \cdot n$. For a proof see \cite{agpr} Theorem 3.4.21.
  
Finally we have a natural action of $\mathfrak G_n(A)$ on $\hom_{\mathcal R}(R, M_n(A))$  by composing a map $f$  with an automorphism $g$. One has a commutative diagram, cf. \eqref{unpp}
\begin{equation}\label{unpp1}
\xymatrix{  R\ar@{->}[r]^{\mathtt \mathrm j_R\qquad}\ar@{->}[d]_f&M_n(T_n(R))\ar@{->}[d]^{M_n(\overline{g\circ f)}}\\
M_n(A)\ar@{->}[r]^g&M_n(A) },\qquad g\circ f= M_n(\overline{  g\circ f})\circ  \mathtt  j_R.
\end{equation}  Assume now that $R$  is an $F$ algebra so also $T_n(R)$ is an $F$ algebra and    $M_n(T_n(R))=M_n(F)\otimes_FT_n(R) $. 

Given  an  automorphism  $g$ of $M_n(F)$,  set $\hat g:=\overline{g\otimes 1\circ \mathtt  j_R}:T_n(R)\to T_n(R)$ so that $$M_n(\hat g) \circ\mathtt  j_R=g\otimes 1 \circ\mathtt  j_R.$$

Given $g_1,g_2$   two automorphisms of $M_n(F)$ we have, when $A=T_n(R)$ and $f= \mathtt  j_R$, in Formula \eqref{unpp1}:
$$   M_n(\widehat {g_1\circ   g_2} ) \circ\mathtt  j_R= (g_1\circ g_2)\otimes 1\circ \mathtt  j_R=g_1\otimes 1\circ  M_n( \hat g_2 ) \circ\mathtt  j_R  $$ $$=   M_n( \hat g_2 ) \circ g_1\otimes 1\circ \mathtt  j_R=  M_n( \hat g_2 )  \circ M_n( \hat g_1 ) \circ \mathtt  j_R.$$ This implies $ \widehat {g_1\circ   g_2}  = \hat g_2  \circ  \hat g_1 $.
so that the map $g\mapsto  \hat g$ is an antihomomorphism from $Aut_{F}(M_n(F))$ to 
 $Aut(T_n(R)).$ Finally
\begin{proposition}\label{inva} The map  $g\mapsto g\otimes \hat g^{-1} $ is a homomorphism from the group  $PGL(n,F)=Aut_FM_n(F)$ to the group of  all automorphisms of $M_n(T_n(R))$. The image of $R$ under $\mathrm j_R$ is formed of invariant elements. 

\end{proposition} 
Theorem \ref{immu} states that when $R$  is an $n$ Cayley--Hamilton algebra the map $\mathrm j_R$ is an {\em isomorphism} to this ring of invariants. We call this statement the {\em strong embedding theorem}.

\section{The trace\label{astr}}
{\em   
In this section we assume, unless otherwise specified, that the algebras are over $\mathbb Q$.}
\subsection{Axioms for a trace\label{axtr}}
 \begin{definition}\label{Traces1}  An associative algebra with trace, over a field
$F$, is an associative $F$ algebra $R$ with a 1-ary operation
\[
t:R\to R
\]
which is assumed to satisfy the
following axioms:
\begin{enumerate}
\item  $t$ is $F$--linear.
\item  $t(a)b=b\,t(a), \quad\forall a,b\in R.$
\item  $t(ab)=t(ba), \quad\forall a,b\in R.$
\item $t(t(a)b)=t( a)t(b), \quad\forall a,b\in R.$
\end{enumerate}

\end{definition}
This operation is called a {\em formal trace}.\index{trace!formal}
We denote $t(R):=\{t(a),\ a\in R\}$ the image of $t$. From the axioms it follows that $t(R)$  is a commutative $F$ algebra which we call the {\em trace algebra of $R$}.\index{algebra!trace}

\begin{remark} We have the following implications:

     Axiom 1) implies that $t(R)$ is an $F$--submodule.

     Axiom 2) implies that $t(R)$ is in the center  of $R$.

     Axiom 3) implies that $t$ is 0 on the space of commutators $[R,R]$.

     Axiom 4) implies that $t(R)$ is an $F$--subalgebra and that $t$ is
$t(R)$--linear.
\end{remark}
These axioms in general are quite weak  and strange traces may appear. For instance,   if $R$  is a commutative $A$ algebra, any $A$ linear map  of $R$ to $A$  satisfies these axioms.

As examples of {\em strange traces},  to use later in order to see why Cayley--Hamilton algebras are special, the reader may use $R=F[x]/(x^2)$ and either $t(1)=0,\ t(x)=1$  or $t(1)=1,\ t(x)=1$.  On the other hand for each $n\in\mathbb N$ the trace $t(1)=n,\ t(x)=0$ is a natural trace arising from matrix theory.

The axioms are in the form of universal algebra so that in the category of algebras with trace free algebras exist, see \S \ref{fta}.\medskip

By a homomorphism of algebras with trace $R_1,R_2$ we mean a homomorphism $f:R_1\to R_2$ of algebras which commutes with the trace, that is $f(t (a))=t (f(a))$. Clearly a  homomorphism of algebras with trace $R_1,R_2$ induces a  homomorphism of their trace algebras $T_1,T_2$.

Thus algebras with trace form a category, which we will denote $\mathcal C_F$  or simply $\mathcal C$ when $F$  is chosen.

In fact   we may distinguish between algebras with 1 or no assumption. In general we make no special assumptions on $t(1)$.

\begin{definition}\label{idege} Given an   algebra $R$ with trace, a set $S\subset R$ generates $R$ if $R$ is the smallest subalgebra closed under trace containing $S$.\smallskip

By {\em trace ideal}  or simply {\em  ideal}  in an algebra $R$ with trace we mean an ideal closed under trace. 
\end{definition}
\begin{remark}\label{trid0} An algebra $R$  can be finitely generated as trace algebra  but not as algebra, for instance the free trace algebra, even in one variable.

Clearly, if $I\subset R$ is   a trace ideal of an algebra with trace, then $R/I$ is an algebra with trace and $t(R/I)=t(R)/[I\cap t(R)]$.   The usual homomorphism Theorems hold  in this case.
\end{remark}

There is a  twist in this definition, clearly if $I,J$ are two trace ideals also  $I+J,\ I\cap J$ are trace  ideals. As for the product  we need to be careful, since $IJ$  need not be closed under trace. For instance, if $I$ is  the ideal of positive elements in the free algebra,   then for each variable $x$ we have $x^2\in I^2$  but $tr(x^2)\notin I^2$.

 So we should define
$$I\cdot J:=IJ+ R\cdot  tr(IJ),\quad I^{\cdot n}:=I\cdot I^{\cdot n-1}.$$  In particular we say that $I$ is nilpotent,  if for some $n$ we have 
$ I^{\cdot n}=0$.

\begin{proposition}\label{fonda}\strut  
\begin{enumerate}\item  Let $R$ be an algebra with trace  and $T$  its trace algebra. If $U$ is a commutative $T$  algebra, then $U\otimes_TR$  is  an algebra with trace  $t(u\otimes r)=u\otimes t(r)$ and $U\otimes 1$ is its trace algebra. 
\item Given two algebras with trace $R_1,R_2$  with trace algebras $T_1,T_2$ their direct sum is  $R_1\oplus R_2$ with trace $t(r_1,r_2):=   (t(r_1),t(r_2)) $      and  trace algebra  $T_1\oplus T_2.$
\item  Finally, if $R$ is an algebra with trace and $T$ is its trace algebra and $T$ has also a trace  with trace algebra $T_1$, then  the composition of the two traces is a trace on $R$,  with trace algebra $T_1$.
\end{enumerate}
\end{proposition}  
In particular we can apply this  when $U=T_S$  is the localization of $T$ at a multiplicative set $S$.
\smallskip

\subsection{The free trace algebra}\label{fta}The free trace algebra over a set $X$ of variables will be denoted by  $\mathcal F_T\langle X\rangle$. By definition it is a trace algebra $\mathcal F_T\langle X\rangle$ containing $X$ and such that for every trace algebra $U$ the set of homomorphisms from $\mathcal F_T\langle X\rangle$ to $U$ are in bijiection with the maps $X\to U$.   It can be described as follows. 

 Start from the usual free algebra  $  F \langle X\rangle$, then consider the 
 classes of cyclic equivalence of monomials $M$, which we formally denote $tr(M)$. The algebra  $\mathcal F_T\langle X\rangle=F\langle x_i\rangle_{ i\in \mathcal I} [tr(M)]$ is the polynomial ring in the infinitely many commuting variables $tr(M)$  over the free algebra          
  $  F \langle X\rangle$. Its trace algebra is the polynomial ring $F[tr(M)]$ in the infinitely many commuting variables $tr(M)$. The map $tr:M\mapsto tr(M)$ is the formal trace.

  \smallskip

As for the usual theory of polynomial identities we have:
\begin{definition}\label{trid} Given a trace algebra $U$, an element $f\in \mathcal F_T\langle X\rangle$ is a {\em trace identity} for $U$, if it vanishes under all evaluations $X\to U$.

It is a {\em pure trace identity} if  $f\in F[tr(M)]$.

\end{definition} 
Usually when speaking of trace identities we will assume that they have coefficients in $\mathbb Q$.

 Then it is easy to see that the set of trace identities of a trace algebra $U$ is a (trace) ideal of  $\mathcal F_T\langle X\rangle$ which is closed  under the endomorphisms (variable substitutions) of $\mathcal F_T\langle X\rangle$.                      

Such an ideal is called a {\em $T$--ideal}. Conversely any $T$--ideal $I$ of   $\mathcal F_T\langle X\rangle$ is the ideal of trace identities of an algebra, namely  $\mathcal F_T\langle X\rangle/I.$ An algebra $\mathcal F_T\langle X\rangle/I $ with $I$ a $T$--ideal, is called a {\em relatively free algebra on $X$},  since it is a free algebra in the {\em variety} of algebras satisfying the identities of $I$. Usually when discussing trace identities one assumes $X$ to be countable. With some care one can also restrict to $X$ finite.

Finally we define:
\begin{definition}\label{treq} Two trace algebras are {\em trace PI--equivalent} of just PI--equivalent, if they satisfy the same trace identities (in countably  many variables).

\end{definition} In particular a trace algebra is PI--equivalent to the free trace algebra modulo the $T$--ideal of its identities.

\section{The Cayley--Hamilton identities}\subsection{One variable  $X=\{x\}$}
We start with the case of just one variable  $X=\{x\}$ which  is of special importance. In this case the trace algebra $\mathcal C_X$ is the polynomial algebra  in the infinitely many  variables $tr(x^i),\ i=0,\ldots,\infty$ while the free algebra $\mathcal F_T\langle X\rangle=\mathcal C_X[x]$.

It is convenient to identify the trace algebra with the {\em ring of symmetric functions on infinitely many variables $\lambda_j$},  with coefficients in $\mathbb Q[tr(1)]$. 

We do this by identifying  $tr(x^i)$  with the power sum  $\psi_j:=\sum_j\lambda_j^i$.  This is compatible with matrix theory, when $x$ is an $n\times n$ matrix over $\mathbb C$ and $\lambda_1,\ldots, \lambda_n$ its eigenvalues we have  $tr(x^i)=\sum_{j=1}^n\lambda_j^i$.

   We have, in the ring $\mathbb Q[tr(x),\ldots, tr(x^i),\ldots ]$  the formal {\em elementary symmetric functions} $\sigma_i(x)$  which correspond in this identification, to the elementary symmetric functions $e_i$. 
   
   The  elementary symmetric functions   are related to the power sums by the recursive formula:
   \begin{equation}\label{daelan}(-1)^m\psi_{m+1}+\sum_{i=1}^{m}(-1)^{i-1}\psi_{i}e_{m+1-i}=(m+1)e_{m+1}\end{equation}
We have the generating formula, for  symmetric functions in $n$   variables $\lambda_j$, obtained 
using the Taylor expansion for $\log(1+y)$.
$$\sum_{i=0}^n(-1)^i e_iu^i=\prod_{r=1}^n(1-\lambda_ru)=
\exp (-\sum_{j=1}^\infty \frac{\psi_j}{j}u^j).$$
 
We then {\em define}  the elements $\sigma_i(x)$ in the free algebra  by
\begin{equation}\label{reldaN1}
\sum_{i=0}^\infty (-1)^i \sigma_i(x)u^i:= 
\exp (-\sum_{j=1}^\infty \frac{tr(x^j)}{j}u^j).\end{equation}
We may define, for each $n$, in the free algebra with trace in a single variable $x$,  the {\em $n$ formal Cayley Hamilton polynomial}
\begin{equation}\label{CH}
CH_n(x):=x^n+\sum_{i=1}^n(-1)^i  \sigma_i(x)x^{n-i}.
\end{equation}
$$\text{Example}\quad CH_1(x)=x-tr(x);\quad CH_2(x)=x^2-tr(x)x+\frac12(tr(x)^2-tr(x^2)).$$
$$
CH_3(x)=x^3 -tr(x)x^2+\frac 1 2(tr(x)^2-tr(x^2))x-\frac 13 tr(x^3)-\frac 1 6 tr(x)^3+\frac 1 2tr(x^2 )tr(x) .
$$
\subsection{The multilinear form}
It is convenient to use also the {\em multilinear} form of the Cayley--Hamilton identity and of the symmetric functions $\sigma_i(x)$, which can be obtained by the process of {\em full polarization}.  For this,  given a permutation $\sigma\in S_m$,
we  decompose  $\sigma=(i_1i_2\dots i_h)\dots (j_1j_2\dots j_\ell)(s_1s_2\dots s_t)$ in cycles
  and  we set:
  \begin{align}\label{phis}&T_\sigma     (x_1,x_2,\dots,x_m)\\
=tr(x_{i_1}x_{i_2}\dots x_{i_h})&\dots tr(x_{j_1}x_{j_2}\dots x_{j_\ell})
tr(x_{s_1}x_{s_2}\dots x_{s_t}).
\end{align}From the basic elements $T_\sigma     $ of Formula \eqref{phis}  take $m=k+1$. We may assume that the last cycle  ends with $s_t=k+1$ so the last factor is of the form $tr((x_{s_1}x_{s_2}\dots x_{s_{t-1}})x_{k+1})$,   hence we have that
\begin{equation}
\label{psis1} T_\sigma     (x_1,x_2,\dots,x_{k+1})=tr(\psi_\sigma(x_1,x_2,\dots,x_k)x_{k+1})\end{equation} where $\psi_\sigma(x_1,x_2,\dots,x_k)$ is the element of   $\mathcal F_T\langle X\rangle$ given by the formula\begin{equation}\label{psis}\begin{matrix}
\psi_\sigma(x_1,x_2,\dots,x_k)\\=tr(x_{i_1}x_{i_2}\dots x_{i_h})\dots tr(x_{j_1}x_{j_2}\dots x_{j_\ell})
 x_{s_1}x_{s_2}\dots x_{s_{t-1}}.
\end{matrix}
\end{equation}
Then we have the multilinear form,  (see also Lew \cite{Lew2}):
\begin{proposition}\label{CH1}\strut   \begin{enumerate}\item For each $k $ the polarized form of $\sigma_k(x)$ is the expression
 \begin{equation}\label{tradet}
T_k(x_1,\dots,x_k)=\sum_{\sigma\in S_k}\epsilon_\sigma T_\sigma     (x_1,\dots,x_k). 
\end{equation}

\item The polarized form of $CH_n(x)$ is  \begin{equation}\label{ilCH}
CH(x_1,\dots,x_n) =(-1)^n\sum_{\sigma\in S_{n+1}}\epsilon_\sigma \psi_\sigma(x_1,x_2,\dots,x_n).
\end{equation} Here $\epsilon_\sigma$ denotes the sign of $\sigma$.
\item \begin{equation}\label{trch}
tr(CH(x_1,\dots,x_n)x_{n+1})=(-1)^n T_{n+1}(x_1,\dots,x_{n+1}). \end{equation} \end{enumerate}

 \end{proposition} For the proof see \cite{depr} or \cite{agpr} Proposition 12.1.12.\smallskip

 {\bf Example $n=2$ (polarize $CH_2(x)$)}
\begin{equation}\label{ch2}
CH_2(x)=x^2-tr(x)x+det(x)=x^2-tr(x)x+\frac 1 2(tr(x)^2-tr(x^2)).
\end{equation}
$$ x_1x_2+x_2x_1-tr(x_1)x_2-tr(x_2)x_1-tr(x_1x_2)+tr(x_1)tr(x_2)$$
$$=CH_2(x_1+x_2)-CH_2(x_1)-CH_2(x_2).$$
  Also from the decomposition into cosets  $S_{n+1}=S_n\bigcup_{i=1}^nS_n(i,n+1)$, of the symmetric group,  one has the recursive formula
  \begin{equation}\label{ref}
T_{n+1}(x_1,\dots,x_{n+1})=T_n(x_1,\dots,x_n)tr(x_{n+1})-\sum_{i=1}^nT_n(x_1,\dots,x_ix_{n+1},\ldots,x_n).
\end{equation}

\subsection{The  first and second fundamental Theorem for matrix invariants}
The first and second fundamental Theorems for matrix invariants    may be viewed as the starting point of the Theory of Cayley--Hamilton algebras. 
 \begin{theorem}\label{SFT}
 The algebra $\mathcal S_n(\ell)$, cf. \ref{genmadt}, of equivariant polynomial maps from $\ell$--tuples of $n\times n$ matrices, $M_n(F)^\ell$  to  $n\times n$ matrices $M_n(F)$,  is the free algebra $\mathcal  F_T\langle X\rangle,\ X=\{x_1,\cdots,x_\ell\}$ with trace modulo the $T$--ideal generated by the $n^{th}$ Cayley Hamilton polynomial and $tr(1)=n$.
  \begin{equation}\label{ffs}
{\boxed{ \mathcal S_n(\ell):=\mathcal  F_T\langle X\rangle/\langle CH_n(x),\ tr(1)=n  \rangle}}.
\end{equation}

\end{theorem}
 
\begin{remark}\label{sez1}
It can be shown that, if we do not set $tr(1)=n$, then we have:
$$ {\boxed{\oplus_{i=1}^n \mathcal S_i(\ell)=\mathcal  F_T\langle X\rangle/\langle CH_n(x)  \rangle}}.$$
\end{remark}
    Let $A_{\ell,n}:=F[\xi_{j,h}^{(i)}]$  denote  the polynomial functions on the space $M_n(F)^\ell $.\label{aelln} 
    That is   the algebra of polynomials over $F$ in $\ell n^2$ variables $\xi_{j,h}^{(i)},$ $ i=1,\ldots \ell;$ $ j,h=1,\ldots,n $. 

On this space, and hence on  $A_{\ell,n}$, acts the group $PGL(n,F)$  by conjugation.

The space of polynomial maps from $M_n(F)^\ell $ to $M_n(F)$ is
$$M_n(A_{\ell,n})=M_n(F)\otimes A_{\ell,n}.$$    On this space acts diagonally $PGL(n,F)$ and the invariants

  $$\mathcal S_n(\ell) : =M_n(A_{\ell,n})^{PGL(n,F)}=(M_n(F)\otimes A_{\ell,n})^{PGL(n,F)}$$
give  the relatively free algebra   in $\ell$ variables in the variety of trace algebras satisfying $CH_n(x)$.

For the trace algebra we have $\mathcal C_n(\ell)=A_{\ell,n} ^{PGL(n,F)}$.  Of course we may let $\ell$ be also infinity (of any type)  and have $$\mathcal S_n(X) =M_n(A_{X,n})^{PGL(n,F)}=(M_n(F)\otimes A_{X,n})^{PGL(n,F)}$$ where $A_{X,n}$ is the polynomial ring on $M_n(F)^X.$
 
For a formulation and proof of these Theorems   in all characteristics or even  $\mathbb Z$--algebras,  the Theorem of Zubkov,  the reader may consult \cite{depr0}.
\section{Cayley--Hamilton  algebras}\subsection{Pseudocharacters and Cayley--Hamilton  algebras}
\begin{definition}\label{CHalg}
An algebra  with trace satisfying the  $n$--Cayley--Hamilton identity  \eqref{ilCH} and $tr(1)=n$ will be called an {\em $n$--Cayley--Hamilton  algebra} or {\em $n$--CH algebra}.  
\end{definition} In other words an $n$--Cayley--Hamilton  algebra is a quotient, as trace algebra, of one relatively free algebra $\mathcal S_n(X)$.

For $n=1$  a  $1$--CH algebra is just a commutative algebra in which the trace is the identity map.

\begin{remark}\label{ests}
If $R$ is an $n$--CH  algebra, and $A$ a commutative $T$ algebra, then $A\otimes_TR$ (\ref{fonda}) is an $n$--CH  algebra.

Assume that $S\subset R$ is a subalgebra closed under  trace, or a quotient $S=R/I$ under a trace ideal $I$.   If $R$ is an $n$--CH  algebra,  then $S$ is also an $n$--CH  algebra.

\end{remark}
\smallskip

Following Taylor \cite{Tay},  we define:\begin{definition}\label{pse}
A  pseudocharacter (or  pseudorepresentation) of a group $G$, of degree $n$ with coefficients in a commutative ring  $A$,  is a map $t :G\to A$  satisfying the following three properties:
\begin{enumerate}\item $t (1)=n$.
\item $t (ab)=t (ba),\ \forall a,b\in G$.
\item  $T_{n+1}(g_1,\dots,g_{n+1})=0,\ \forall g_i\in G$ (Formula \eqref{tradet}).\end{enumerate}
\end{definition}
Frobenius \cite{Frob}, discovered already  that this is a property of an $n$--dimen\-sio\-nal character.

The connection between the previous two definitions is the following.   One considers  the group algebra $A[G]$  and then extends the map $t $  to  a trace with trace algebra $A$. Next one considers the {\em Kernel} of the trace, that is
$$K_t:=\{a\in A[G]\mid t (ab)=0,\ \forall b\in A[G]\}.$$
It is then an easy fact to see that, if $t $  is a pseudocharacter  of $G$ of degree $n$, then $A[G]/K_t$ is a $n$-- Cayley Hamilton algebra. In particular  when $A\supset \mathbb Q$  one can apply Theorem \ref{immu}. In general we have:

\begin{proposition}\label{ech}
If $R$ is an algebra with trace  and the trace satisfies $$
T_{n+1}(x_1,\dots,x_{n+1})=\sum_{\sigma\in S_{n+1}}\epsilon_\sigma T_\sigma     (x_1,\dots,x_{n+1})  
 =0 ,$$ then $CH_n(x_1,\dots,x_n)\in K_t$ and $R/K_t$ is an $n$-- Cayley Hamilton algebra.
\end{proposition} Observe that, from Formula \eqref{ref} it follows that, if $R$  satisfies the multilinear identity $T_{n+1}$,  then it satisfies $T_j,\ \forall j\geq n+1$.

From the next Theorem \ref{immu} then follows:
\begin{theorem}\label{Gimmu}
Given a  pseudocharacter $t$  of a group $G$, of degree $n$ with coefficients in a commutative ring  $A\supset\mathbb Q$ there is a commutative $A$--algebra   $B$ and a representation $\rho: G\to GL_n(B)$  so that:
\begin{equation}\label{carr}
t(g)=tr(\rho(g)),\ \forall g\in G,\quad \ker(\rho)=K_t, \ \rho:A[G]\to M_n(B)
\end{equation} with $K_t$ the kernel of the associated trace form.
\end{theorem}
 In particular in  \cite{Tay} the Theorem that, if $A$ is an algebraically closed field of characteristic 0, then the pseudocharacter is associated to a unique semisimple representation of $G$, follows from our Theorem \ref{ssre1} 2).

\subsection{The main Theorem of Cayley--Hamilton algebras}

 Let now $R$  be any  $n$--CH algebra over  a field $F\supset\mathbb Q$. 
 In Proposition \ref{laTR} the universal map into $n\times n$ matrices   $\mathrm j_R:R\to M_n(T_n(R))$ is  assumed to be trace preserving.  Then by the discussion of \S \ref{symme} the map $\mathrm j_R$  maps $R$ into the subalgebra of $PGL(n)$  invariants. The main Theorem is,  \cite{P5} or \cite{agpr} Theorem 14.2.1

 \begin{theorem}\label{immu}
 
 The map  $\mathrm j_R:R\to M_n(T_n(R))^{PGL(n,\mathbb Q)}$ is an isomorphism.
 \end{theorem} 
  
   \begin{definition}\label{cattn}
 We denote by  $\mathfrak T_n$  the category  of commutative algebras equipped with a rational  $PGL(n,F)$ action (and where the morphisms are   $PGL(n,F)$ equivariant). 

\end{definition}
 The functor $R\mapsto T_n(R)$ is from the category of $n$--CH algebras  to the category $\mathfrak T_n$.

Since $PGL(n,F)$ acts both on $M_n(F)$ and on $C$ by automorphisms one has the diagonal action  of  $PGL(n,F)$ on $M_n(C)$. 
 This functor  has a right adjoint  $C\mapsto R_n(C):=M_n[C]^{PGL(n,F)}$:
  \begin{equation}\label{rag}
\boxed{\hom_{\mathfrak T_n}(T_n(R), C)\simeq \hom_{\mathcal C}(R, R_n(C))},\quad \phi\mapsto 1\otimes \phi, \ $$$$\begin{CD}
R=(M_n(F)\otimes T_n(R))^{PGL(n,F)}@>1\otimes \phi >>(M_n(F)\otimes C )^{PGL(n,F)}=R_n(C).
\end{CD}.
\end{equation}
 
    In other words, always for $\mathbb Q$ algebras:
\begin{proposition}\label{ful}
 The functor $R\mapsto T_n(R)$ is an equivalence   between   the category of $n$--CH algebras  and a {\em full} subcategory of the category $\mathfrak T_n$.
\end{proposition}   
It is easy to give examples of algebras in $\mathfrak T_n$  which are not  of the form $T_n(R)$. For instance take a nilpotent conjugacy class $\mathcal O$ of $n\times n$ matrices.   

 If   the order of nilpotency  of elements of $\mathcal O$ is $k\leq n$,  then one can prove  that $R_n(C)=F[x]/(x^k),\ tr(x^i)=0,\ i=1,\ldots,k-1$.  This is independent of the conjugacy class but depends only on $k$.\smallskip

In general the nature of $T_n(R)$  may be quite difficult to study. 

For instance consider the {\em scheme of pairs of commuting matrices} that is the commutative ring $A_n:=F[X,Y]/(XY-YX=0)$ generated by the entries of two generic $n\times n$ matrices $X,Y$  modulo the ideal generate  by the entries of $XY-YX$. One easily sees that it  is of the form $T_n(R),$ where $  R= F[x,y]   $ is the commutative polynomial algebra in two variables.   It is not known if this ring is a domain.  

An important consequence of   Theorem \ref{immu} is:
\begin{corollary}\label{PIn}
An   $n$--CH algebra $R$ satisfies all  polynomial identities of $M_n(\mathbb Q)$ with coefficients in $\mathbb Q$.
\end{corollary}    
 \subsection{Azumaya algebras} An important class of CH--algebras are Azumaya algebras, \cite{Az}, \cite{A-G2}.  For our purpose we may take as definition:
\begin{definition}\label{azza}
An   algebra  $R$ with   center $Z$ is Azumaya of rank $n^2$ over   $Z$, if there is a faithfully flat  extension $Z\to W$ so that $W\otimes_ZR\simeq M_n(W)$.
\end{definition} Then, if $a\in R$,  all the coefficients of the characteristic polynomial of $a\otimes 1\in M_n(W)$ are in $R$, and they are independent of the splitting $M_n(W)$. 
\begin{definition}\label{redt0}
In particular the trace of $a\otimes 1$ is called the {\em reduced trace of $a$}, denoted $tr(a)$.
\end{definition}\medskip

Azumaya algebras play a special role in this Theory.  In particular let us recall a theorem of M. Artin  \cite{ArM}, as generalized by Procesi  \cite{P1}.
 
 \begin{theorem}\label{ArT}[Artin--Procesi]
A ring  $S$ is a rank $n^2$ Azumaya algebra over its center $A$, if and only if  $S$ satisfies all polynomial identities of $n\times n$ matrices and no quotient of $S$ satisfies the   polynomial identities of $n-1\times n-1$ matrices.
\end{theorem} The strength of this Theorem lies in the fact that no a priori hypotheses on its center are needed.

 If  $R$ is an Azumaya algebras   of rank $n^2$ over its center $Z $, then one can prove that $T_n(R)$ is faithfully flat over $Z$  and  $R\otimes_ZT_n(R)=M_n(T_n(R)$. 
 
 We think of $R$ as a {\em non split form of matrices}, see \cite{agpr},  \S 10.4.1..                      
 We claim that
 \begin{proposition}\label{inttr}    If an Azumaya algebras $R$  of rank $n^2$ over its center $Z\supset\mathbb Q$ has a $Z$--linear trace  $t$ with respect to which it is  an $n$--CH  algebra, then $t=tr$ the usual reduced trace.

\end{proposition}
\begin{proof}
 
By the main Theorem  \ref{immu} we have  a trace preserving embedding of $R$  in the universal algebra $M_n(T_n(R))$ under which $t$ is the trace and, since the reduced trace is independent of the splitting, the claim follows.\end{proof}                  
In fact   it also follows that, under the same hypotheses 
 with respect to $t$, if  we have that $R$ is  an $k$--CH  algebra, then $k=i\cdot n$ for some $i$ and     $t= i\cdot tr$  with $tr$ the usual reduced trace.\smallskip

\section{The variety of semisimple representations\index{representation!variety of semisimple}\label{varssr}}
\subsection{The Geometric quotient}
 There is a geometric interpretation of Theorem \ref{FFT}. Consider the space of $\ell$--tuples of matrices
 $M_n(F)^\ell$  where   we shall now assume that $F$ is algebraically closed of any characteristic.

  We think of the space $M_n(F)^\ell$  as {\em the space of $n$--dimensional representations of the free algebra $F\langle x_1,\ldots,x_\ell \rangle$  in $\ell$ generators}, where a representation $\rho:F\langle x_1,\ldots,x_\ell \rangle\to M_n(F)$ corresponds to the $\ell$--tuple $(a_i:=\rho(x_i))$.

   The linear group  $GL(n,F)$ acts by simultaneous conjugation, in fact this action is trivial by the scalar matrices and hence should be thought of as an action of the {\em projective linear group}\index{projective linear group}
 $$G:=GL(n,F)/F^*=PGL(n,F).$$ \begin{remark}\label{isoclasses}
  Clearly two $\ell$--tuples are in the same orbit, if and only if the two representations are isomorphic.

In other words the study of isomorphism classes of representations is the same as the study of $G$--orbits.

\end{remark}
From the definitions it also follows that the $n$--dimensional representation $\rho:  F\langle x_1,\ldots,x_\ell \rangle\to M_n(F)$ can   be thought of as a representation of   the relatively free CH--algebra, Theorem \ref{SFT} , $\rho:\mathcal S_n(\ell)\to M_n(F)$ compatible with the trace.  We use the notations of that Theorem.

 By Geometric invariant theory  the variety associated to the  trace algebra   $\mathcal C_n(\ell)=A_{\ell,n} ^{PGL(n,F)}$ of the algebra $\mathcal S_n(\ell)$  parametrizes closed orbits, hence  we should understand   which representations correspond to closed orbits.  
 
 This has been shown by M. Artin\index{Artin}, \cite{ArM}, in the same paper in which he proved Theorem \ref{ArT}, see \cite{agpr}  Chapter 14.
 
 \begin{theorem}\label{sesisc}  A semisimple representation associated to any given representation is in the closure of its orbit.

A representation is in a closed orbit, if and only if it is semisimple.
\end{theorem}
Denote by 
$$\pi:M_n(F)^\ell \to  V_n(\ell):=M_n(F)^\ell //GL(n,F)$$  the quotient map associated to the inclusion $\mathcal C_n(\ell)=A_{\ell,n} ^{PGL(n,F)}\subset A_{\ell,n}$.
\begin{theorem}\label{dimTr}\strut  \begin{enumerate} \item 
If $\ell\geq 2$, then  the quotient variety $V_n(\ell)$ is of dimension $(\ell-1)n^2+1$. 
\item 
The points of  $V_n(\ell)$ parametrize isomorphism classes of semisimple representations of dimension $n$ of the free algebra in $\ell$--variables. 

Equivalently of trace compatible semisimple representations of  $\mathcal S_n(\ell)$.\end{enumerate}
\end{theorem}
\subsection{The  non commutative algebras} Assume now $F$  is algebraically closed  and of characteristic 0, (for the general case see \cite{P2},  \cite{P6}). We use some results from the next Part II.\smallskip

The algebra $\mathcal S_n(\ell)$ also has a geometric interpretation.  

Consider a maximal ideal $\mathfrak m\subset \mathcal C_n(\ell)$ corresponding to a  point $p\in V_n(\ell)$. By Proposition \ref{tradi} we have 
$$\mathfrak m \mathcal S_n(\ell)\cap  \mathcal C_n(\ell)=\mathfrak m$$ and the  CH  algebra $  \mathcal S_n(\ell)/ \mathfrak m \mathcal S_n(\ell)$ is finite dimensional with trace algebra $F= \mathcal C_n(\ell)/ \mathfrak m.$  

The points of  $M_n(F)^\ell $ thought of as   representations  of  $\mathcal S_n(\ell)\to M_n(F)$ which factor through $$ \mathfrak A_p:= \mathcal S_n(\ell)/ \mathfrak m \mathcal S_n(\ell)=M_n(A_{n,\ell}/\mathfrak m )^{GL(n,F)}$$  are thus the points in the fiber $\pi^{-1}(p)$. In particular, the radical $J$  of $ \Sigma_p$ vanishes on the closed orbit  formed by semisimple representations   $ \mathfrak A_p$.  

By Corollary \ref{CHC1} $J$   is in fact the kernel of the trace form. The algebra $\bar{ \mathfrak A}_p $   is a semisimple algebra  with trace and simple as trace algebra. 

So, given a point $q$ in the closed orbit the corresponding representation   $\rho_q: \overline { \mathfrak A}_p\to M_n(F)$ is injective.  Now a semisimple subalgebra    of $M_n(F)$ is described by Proposition \ref{ssn}.  We then have the finitely many strata of  closed orbits associated to the  lists $m_1,\ldots,m_k$ and $a_1,\ldots,a_k$ of positive integers with $\sum_jm_ja_j=n$.

As we shall remark later, \S \ref{geoti}, in the quotient variety, these are the smooth strata of Luna's stratification by stabilizer type.\smallskip

\subsection{The  smooth part} Of particular importance is the open  set $U_{n,\ell}$ of {\em irreducible representations} that is of $\ell$--tuples of matrices which generate the algebra $M_n(F)$ ($\ell\geq 2$).  On this open set  the group $PGL(n,F)$ acts freely and the quotient $U_{n,\ell}//PGL(n,F)$  is smooth in $M_n(F)^\ell //GL(n,F)$, in fact except a few cases this is exactly the smooth part of $M_n(F)^\ell //GL(n,F)$. 

The map $\tilde\pi:U_{n,\ell}\to U_{n,\ell}//PGL(n,F)$ is a principal $PGL(n,F)$ bundle locally trivial in the \'etale topology, see \cite{Luna}.

This geometric description has a counterpart in the structure of the algebra $\mathcal S_n(\ell)$ which is an Azumaya algebra of rank $n^2$  over its center exactly in the points of $U_{n,\ell}//PGL(n,F)$. 

This fact can be seen in a more explicit form as follows. Let $p\in U_{n,\ell} $  be a point corresponding to a surjective map  $\rho:   \mathcal S_n(\ell)\to   M_n(F)$. Then there are two elements  $a,b\in   \mathcal S_n(\ell)$ so that $\rho(a)$  is a diagonal matrix with distinct and non zero entries and $\rho(b)$ a full cycle on the canonical basis.  Then the elements $\rho(a)^i\rho(b)^j,\ i,j=0,\ldots,n-1$ form a basis of  $M_n(F)$ (see also Lemma \ref{dueg}). 

  Then the two invariants  $D(  a)$,  the discriminant of $  a$,  and   the discriminant of the  basis $  b ^i   a ^j$, $\Delta:=\det( tr(  b ^i   a ^j))$,  are in $\mathcal C_n(\ell)$  and do not vanish  on $p$, and hence on $\pi(p)$. \smallskip

 Let $E_p:=\mathcal C_n(\ell)[D(  a)^{-1}, \Delta^{-1}][t]/CH_n(a)(t),$  where by $CH_n(a)(t)$ we denote the characteristic polynomial of $a$.
 
 We claim that $E_p$ is \'etale over 
$\mathcal C_n(\ell)[D(  a)^{-1}, \Delta^{-1}]$ and
$$E_p\otimes_T \mathcal S_n(\ell)\simeq M_n(E_p).$$
In fact   inverting $\Delta$  implies that $\mathcal S_n(\ell)[\Delta^{-1}]$  is a free module over $\mathcal C_n(\ell)[\Delta^{-1}]$ with basis $  b ^i   a ^j,\ i,j=0,\ldots,n-1$.  Then inverting also   $D(  a)$ gives that the subalgebra $\mathcal C_n(\ell)[(D(  a)\Delta)^{-1}][a]\subset \mathcal S_n(\ell)[(D(  a)\Delta)^{-1}]$ is isomorphic to $E_p$.  

Since we have inverted $D(a)$, then the  characteristic polynomial of the element $a$ has distinct eigenvalues so that adding $a$ is a {\em simple \`etale  extension, cf. \cite{Ray}}. 

The left multiplication of $R:=\mathcal S_n(\ell)[(D(  a)\Delta)^{-1}]$ on itself maps $R$ isomorphically to $End_{E_p}  (R)$, since $R=\mathcal S_n(\ell)[(D(  a)\Delta)^{-1}]$ is a free $E_p$  module  with basis the elements $  b^j,\ j=0,\ldots, n-1.$  Therefore, by definition, $R =\mathcal S_n(\ell)[(D(  a)\Delta)^{-1}]$ is a rank $n^2$ Azumaya algebra  over 
$\mathcal C_n(\ell)[D(  a)^{-1}, \Delta^{-1}]$.\medskip

By the compactness of the Zariski topology  one has a finite  covering of $U_{n,\ell}//PGL(n,F)$  by affine open sets  associated to pairs $  a_i,  b_i$
 as above. \smallskip
 
 Notice that the fibration $\pi:U_{n,\ell}\to U_{n,\ell}//PGL(n,F)$ is NOT  locally trivial in the Zariski topology,  since the ring of fractions of 
 $\mathcal S_n(\ell)$ is a division algebra and not a matrix algebra over a field, as it would be if  locally trivial in the Zariski topology.\vfill\eject

\part{General theory}
\section{Cayley Hamilton algebras\label{leCH}}

 \subsection{Generalities on trace algebras}
 
\begin{definition}\label{idita}\strut  
\begin{enumerate}\item A  {\em simple} trace algebra is one with no proper trace ideals, 
\item A  {\em prime} trace algebra is one in which, if $I,J$ are two trace ideals with $IJ=0$, then either $I=0$ or $J=0$.

\item Finally a   {\em semiprime} trace algebra is one in which, if $I $  is an ideal  with $I^2=0$, then  $I=0$.
\end{enumerate}

\end{definition} Notice that prime implies semiprime.

\begin{definition}\label{ker}
Given a trace algebra $R$ the set
\begin{equation}\label{Ker}
K_R:=\{x\in R\mid t(xy)=0,\ \forall y\in R\}
\end{equation} will be called the {\em kernel} of the trace algebra.

$R$ is called {\em nondegenerate} if $K_R=0$.\smallskip

 Given  a (trace ideal) $I$ in a trace algebra $R$, we set $K(I)\supset I$  to be the ideal such that $R/K(I)=K_{R/I}$.
We call $K(I)$  the {\em radical kernel} of $I$.
\end{definition}

\begin{proposition}\label{tradi} Let $R$  be an algebra with trace $tr$ and $T$  its trace ring. Assume that $tr(1)$ is invertible in $T$,  then:
\begin{enumerate}\item Given any ideal $I$ of $T$  we have that $IR$ is a trace ideal and $IR\cap T=I$ so $R/IR$ is an algebra with trace and trace ring $T/I$.
\item Moreover  $R$  decomposes into the direct sum  $T\oplus R^0$, of $T$ modules,  with $R^0$ the space of trace 0 elements.

 \item If   $R$ is prime, resp. simple, as algebra with trace, then $T$ is a domain, resp. a field.
\end{enumerate}  

\end{proposition}
\begin{proof}
1) Let $a=\sum_jt_jr_j\in R$  with $t_j\in I,\ r_j\in R$. Taking traces we have
$$a\cdot tr(1)=tr(a)=\sum_jt_jtr(r_j),  \implies a= tr(1) ^{-1}\sum_jt_jtr(r_j)\in I.$$
2) The second part follows from axiom (4)  as the map $x\mapsto t(1)^{-1}t(x)$ is a $T$  linear projection to $T$ with kernel $R^0$.

3) We assume $tr(1)$ is invertible and $T\neq\{0\}$. If $a,b\in T$ are non zero and $ab=0$,  then  $aR, bR$  are two trace ideals and $aRbR=0$ a contradiction.

If $R$ is simple and $a\in T,\ a\neq 0$, then $aR$ is a trace ideal hence $aR=R$. So there is a $b\in R$ with $ab=1$  and then $a\cdot tr(b)=tr(1)$. Since $tr(1)$ is invertible the claim follows. 
\end{proof}

 Observe that an algebra $R$  can be considered as algebra with trace by setting the trace identically equal to 0.
 \subsection{ CH  algebras}

\begin{proposition}\label{nonde}
$K_R$ is the maximal trace ideal $J$ where $tr(J)=0$. 

If $t(1)$ is invertible, then  $R/K_R$ is non degenerate.

If  $R$ is  an $n^{th}$--CH  algebra, then we have $K_R^{n^2}=0.$\footnote{$n^2$ is not the best bound, conjecturally the best is $\binom {n+1}2$. But this has been  verified only for very small values of $n$.}
\end{proposition}
\begin{proof}
The first part is clear. As for the second, if $a\in R$ is in the kernel modulo $K_R$, then  we have that, for all $r\in R,\ t(ar)\in K_R$.  So $$t(t(ar))=t(ar)t(1)=0,\ \implies t(ar)=0$$ and the claim follows.

As for the last statement, we have that the   Cayley Hamilton identity on $I$ is $x^n=0$ so the statement follows from Razmyslov's estimate, in the so called Dubnov--Ivanov Nagata--Higman Theorem, \cite{agpr} Theorem 12.2.13.
\end{proof}

\begin{lemma}\label{duni} Let $a$ be an $n\times n$  matrix with entries in a commutative ring $A$. \smallskip

 1)\quad   If $a$ is nilpotent, then 
also $tr(a)$ is nilpotent.

2)\quad If $a=a^2$ is   idempotent,    then $tr(a)$  satisfies the monic polynomial with integer coefficients $\prod_{i=0}^n(x-i).$

\end{lemma}
 \begin{proof}
1)\quad Since in a commutative ring $A$ the set of nilpotent elements is the intersection of the
prime ideals we are reduced to the case in which $A$ is a domain. Hence we can embed it into a
field. For matrices over a field the trace of a nilpotent matrix is 0.

2)\quad We may assume that $A$ has a 1. From   the localization  principle it is enough to prove this for $A$ local. In this case $A^n=aA^n\oplus (1-a)A^n$. The two projective  modules of this decomposition are both free and the rank of $aA^n$ is some integer $0\leq i\leq n$ so  $tr(a)=i$ and the claim follows.\end{proof}
\begin{proposition}\label{CHC} Let $R$ be a $n^{th}$--CH  algebra
 and $r\in R$ a nilpotent element, then $tr(r)$ is nilpotent.
\end{proposition}
 \begin{proof}
By Theorem \ref{immu} we can embed  $R$ into matrices over a commutative ring so that the trace
becomes the ordinary trace. Hence the statement follows from  Lemma \ref{duni}.\end{proof}
\begin{corollary}\label{CHC1} Let $R$ be a $n^{th}$--CH  algebra with trace algebra reduced (no nonzero nilpotent elements). 
  Then, if $r\in R$ is nilpotent,  we have $r^n=0$.  The kernel $K_R$ is the maximal nil ideal and $K_R^{n^2}=0$.
\end{corollary}
In particular we have 
\begin{corollary}\label{CHC2}   1)\quad An  $n^{th}$--CH  algebra $R$ is semiprime, if and only if,  its   trace algebra is reduced  and the  kernel $K_R=0$.

2)\quad An  $n^{th}$--CH  algebra $R$ is   prime, if and only if,   its   trace algebra is a domain and $K_R=0$.
\end{corollary}
\begin{proof}
Assume $R$ semiprime. If the trace algebra contains a non zero nilpotent element, then it contains one with $a^2=0$  and   $Ra$ is a trace ideal with $(Ra)^2=0$.  Also, if $K_R\neq 0$, since $K_R^{n^2}=0$  the algebra is not semiprime.

Conversely, if $R$  has an ideal $I\neq 0$  with $I^2=0$,  then for each $a\in I$ we have $tr(a)$  is nilpotent, by Proposition \ref{CHC} hence $tr(a)=0$ by assumption. If for all $a\in I$ we have $tr(a)=0$, then  $I\subset K_R$ hence $I=0$.\smallskip

As for the second statement let  us show that the given conditions imply $R$  prime. In fact  given two ideals  $I,J$ with $IJ=0$ since  $t(I)\subset I,\ t(J)\subset J$ we have $t(I)t(J)=0$. Since these are ideals, and $T$ is a domain, then one of them must be 0.  If $t(I)=0$, then $I\subset K_R=\{0\}$  by hypothesis.

Conversely if $R$ is prime it is also semiprime, then  we must have $K_R=0$. By Proposition \ref{tradi} 4)    $ T(R)$  is a domain.
\end{proof} Finally the local finiteness property:
\begin{proposition}\label{finl}
An  $n^{th}$--CH  algebra $R$ finitely generated over its trace ring $T$  is a finite $T$ module.
\end{proposition}
\begin{proof}
The Cayley Hamilton identity implies that each element of $R$ is integral over $T$ of degree $\leq n$ then this is a standard result in PI rings  consequence of Shirshov's Lemma, see \cite{agpr} 
Theorem 8.2.1..
\end{proof}

\begin{lemma}\label{multest}
Let $R\subset S$ be an inclusion of trace algebras over some field  $F$. Assume that $R\cdot T(S)=S$. Then $R$ and $S$ satisify the same trace identities.
\end{lemma}
\begin{proof}
Since  we are in characteristic 0, by polarization it is enough to prove that every multilinear trace identity of $R$ holds in $S$. This follows from the fact that,  by definition, the trace of $S$ is $T(S)$ linear.
\end{proof}
\begin{theorem}\label{pieqf}
 Let $R$ be an $n$--CH  algebra  over a field $F$, then $R$ is trace--PI equivalent to a finite dimensional $F$  algebra.
\end{theorem}
\begin{proof} First $R$ is PI--equivalent to its associated relatively free algebra $\mathcal F_R(X)$ which is a quotient of the free Cayley--Hamilton algebra $  \mathcal S_n(X)$ associated to matrices. By Capelli's theory this last algebra decomposes under the linear group (acting on the space spanned by $X$) in irreducible representations of height $\leq n^2$ (cf. \cite{agpr}  Theorem 3.6).  Thus the same is true for  $\mathcal F_R(X)$ which is thus PI equivalent to the algebra on the first $n^2$ variables. We may thus assume that $R$ is finitely generated over $F$.\label{Capelli}

By Theorem \ref{immu} we embed  $R\subset M_n(A)$,  with $A$  a  commutative algebra, finitely generated over $F$, and the embedding is compatible with the trace.  Then $A$  can be embedded in a commutative algebra $B$  which contains a field $G\supset F$  and it is finite dimensional over $G$, see Cohen \cite{cohen}.  

Thus $S:=RB\subset M_n(B)$ is a trace algebra, finite dimensional over $G$,  and satisfies the same trace identities with coefficients in $F$  as $R$ by Lemma \ref{multest}.

By enlarging $G$,  if necessary, we may assume that $  S=S/J=\oplus_iM_{n_i}(G)$ and $J=\oplus_jGa_j$, for some finite set of elements $a_j$, and $J^h=0$ for some $h$. Then the algebra $\tilde R$, generated by $\oplus_iM_{n_i}(F)$ and the elements $a_i$, is finite dimensional.  We claim that the trace algebra of $\tilde R$ is also finite dimensional. Since trace is linear, it is enough to show that $tr(a)$ is algebraic over $F$ for $a$ in a basis of  $\tilde R$ over $F$.  

Now the traces of the nilpotent elements, in particular of the elements in its radical, are all nilpotent; so we only need to show that the traces of the elements $e_{i,i}$  are algebraic over $F$.  

Now, if $e=e^2$  is an idempotent, then $tr(e)$ satisfies the  polynomial of Lemma \ref{duni}. Finally the algebra $\mathfrak R$ generated by $\tilde R$ and its traces is finite dimensional and,  since $\mathfrak RG=S$, it is PI--equivalent to $S$ and hence to $R$,  by Lemma \ref{multest}.\end{proof}\medskip

\section{Semisimple algebras}
\subsection{Semisimple matrix algebras}
 Given   two lists $\underline m:=m_1,\ldots,m_k$ and $\underline a:=a_1,\ldots,a_k$ of positive integers with $\sum_jm_ja_j=n$  consider the algebra  \begin{equation}\label{Fma}
F(\underline m ;\underline a ):=\oplus_{i=1}^k M_{m_i}(F),\ \text{  with trace } 
  t(r_1,\ldots,r_k)=\sum_{i=1}^k tr(r_i)a_i,\end{equation} and $tr(r_i)$ the trace as matrix.

$F(\underline m ;\underline a )$ is a trace subalgebra (of block diagonal matrices) of $M_n(F)$, with the block $M_{m_i}(F)$ repeated $a_i$ times. Hence $F(\underline m ;\underline a )$  is an  $n$--CH algebra, and, as trace algebra, it is {\em simple}.

  Conversely  
  \begin{proposition}\label{ssn}
If $F$ is algebraically closed and $S\subset M_n(F)$ is a semisimple algebra,  then $S$ is one of  the algebras $F(m_1,\ldots,m_k;a_1,\ldots,a_k)$.
\end{proposition}
\begin{proof} 
A semisimple algebra $S$  over $F$ is of the form $S=\oplus_{i=1}^k M_{m_i}(F)$. 

An embedding of $S$  in $M_n(F)$ is a faithful $n$--dimensional representation of $S$. Now the  representations of $S$  are direct sums  of the irreducible representations $ F^{m_i} $ of the blocks  $M_{m_i}(F)$, and a faithful $n$--dimensional representation of $S$ is thus of the form
$$\oplus_i (F^{m_i})^{\oplus a_i},\ a_i\in\mathbb N,\ a_i>0,\ \sum_i a_im_i=n. $$
For this representation the algebra $S$  appears as block diagonal matrices, with an $m_i\times m_i$ block  repeated $a_i$ times. The trace is then the one described in Formula \eqref{Fma}.

\end{proof}
 
\begin{theorem}\label{simc}
Let $F$ be an algebraically closed field of characteristic 0 and $S$  an $n$--CH algebra with trace algebra  $F$ and kernel $K_S$. 

Then $S/K_S$ is finite dimensional, simple,  and isomorphic to one of the algebras  $F(m_1,\ldots,m_k;a_1,\ldots,a_k)$.
\end{theorem} 
\begin{proof} Passing to $S/K_S$ we may assume that $K_S=0$. 
Let us first assume that $S$ is finite dimensional, then by Corollary \ref{CHC1}   we   have that $S$ is a semisimple algebra so it is of the form $S=\oplus_{i=1}^k M_{m_i}(F)$. 

Since  $S$  is an $n$--CH algebra, then it is a quotient of one of the free algebras  $S_n(m)$ for some $m$. Restricted to the trace algebra $T_n(m)$  this gives a point  $p$ in the quotient variety, thus  $S$, as trace algebra, is of the form $\overline \Sigma_p$  and the statement follows from the   discussion of \S \ref{varssr}.

Now let us show that $S$ is finite dimensional. For any choice of  a finite set of elements  $A=\{a_1,\ldots,a_k\}\subset S$ let $S_A$ be the subalgebra generated by these elements. By  Proposition \ref{finl} each $S_A$ is finite dimensional. Then, if $J_A$ is the radical of $S_A$,  we have, by the previous discussion,  that $\dim S_A/J_A\leq n^2$.  Let us  choose $A$ so that  $\dim S_A/J_A$ is maximal. We claim that $S=S_A$ and $J_A=0$.

We have that $J_A$ is the kernel of the trace form of $S_A$.

First let us show that $J_A=0$.   Since $K_S=0$ it is enough to show that, if $a\in J_A$  and $r\in S$, then  we have $tr(ra)=0$. If $r\in S_A,\ ra\in J_A$ and this   follows from Lemma \ref{duni} 1).  If $r\notin S_A$, then $S_{A,r}\supsetneq S_A$ and we claim that $J_{A,r}\supset  J_A$, in fact   otherwise $\dim S_{A,r}/J_{A,r}> \dim S_A/J_A$ a contradiction.

Then, by the previous argument $tr(ar)=0$ so $a\in K_S =0$ and $J_A=0$. Next, if $S_A\neq S$, then  we have again some $S_{A,r}\supsetneq S_A$ and now $J_{A,r}\neq 0$ a contradiction.
\end{proof}

\section{The trace algebra is a field}

We want to study general CH algebras over a field $F$ such that the values of the trace are in  $F$.
\subsection{The classical case}
First  some examples.
\begin{example}\label{esf}
If $R$  is a finite dimensional simple  $F$  algebra, then we have $R=M_k(D)$  with $D$  a division ring, finite dimensional  over its center $G$,  which is also finite dimensional  over $F$. Let  $\dim_GD=h^2,\ \dim_FG=\ell$.  

The algebra $R$  is endowed with a canonical {\em reduced trace}  which for $F=G$  coincides with that of Definition \ref{redt0} (since $D$ is Azumaya over $G$ of rank $(hk)^n$).  This trace is a composition of two traces
$$tr_{R/F}= tr_{G/F}\circ  tr_{R/G}.$$ For the   reduced trace $tr_{R/G}$, according to  \ref{azza}, we take an algebraic closure $\overline G$ of $G$  then, if $a\in M_k(D)$: $$\dim_GD=h^2\implies M_k(D)\otimes_G\overline G=M_{k\cdot h}(\overline G)$$  and the trace  $tr_{R/G}(a):=tr(a\otimes 1) $ as matrix.
\end{example}

We already remarked  that  $tr_{R/G}(a)\in G$.  As for $tr_{G/F}(g),\ g\in G$  one takes the trace of the multiplication by $g$  a $\ell\times\ell$ matrix over $F$.

If the characteristic of $F$ is 0 (or in general if $G$ is separable over $F$), then we have 
$$G\otimes_F\overline G= \overline G^\ell,\  g\otimes 1=(\lambda_1,\ldots,\lambda_\ell)\implies tr(g)=\sum_i\lambda_i.$$

\begin{proposition}\label{redt}
If the characteristic of $F$ is 0, then  $R$ is a $k\cdot h\cdot \ell$ CH algebra.
\end{proposition}
\begin{proof}
We have  
$$ M_k(D)\otimes_F\overline G=M_k(D)\otimes_G(G\otimes_F\overline G)=M_k(D)\otimes_G( \overline G^\ell)=M_{k\cdot h}(\overline G)^\ell\subset M_{k\cdot h\cdot \ell}(\overline G)$$ and the reduced trace is induced by the trace of $M_{k\cdot h\cdot \ell}(\overline G)$.
\end{proof}
 
Consider a  general semisimple algebra finite dimensional over $F$
$$R=\oplus_{i=1}^pM_{k_i}(D_i),\  \dim_{G_i}D_i=h_i^2,\ \dim_FG_i=\ell_i$$ where $G_i$ is the center of the division algebra $D_i$.  Given positive integers  $a_i,\ i=1,\ldots, p$  we may define the trace
\begin{equation}\label{trconp}
t(r_1,\ldots,r_p):=\sum_{i=1}^pa_itr(r_i),\ r_i\in M_{k_i}(D_i),\ tr(r_i)\in F\quad\text{the reduced trace}
\end{equation}
\subsection{The uniqueness Theorem}
\begin{theorem}\label{trach} The algebra $R=\oplus_{i=1}^pM_{k_i}(D_i)$ with the previous trace is an $n$--CH algebra with $n=\sum_ia_in_i,\ n_i=k_ih_i\ell_i$.

Conversely any  trace on $R$  which makes it into an $n$--CH algebra for some $n$ is of the previous form.

\end{theorem}
\begin{proof}
In one direction the statement is clear. The reduced trace $tr(r_i)$  is the ordinary trace associated to an embedding of  $M_{k_i}(D_i)$  into $n_i\times n_i$ matrices over the algebraic closure $\overline F$.  So  the trace of formula  \eqref{trconp} is the ordinary trace associated to an embedding of  $R$  into $n \times n $ matrices over $\overline F$.

For the second,   such a trace    induces  a trace in 
 \begin{equation}\label{sspa}
R\otimes_F\overline F=\oplus_{i=1}^pM_{k_i}(D_i)\otimes_F\overline F=\oplus_{i=1}^pM_{k_i\cdot h_i}(\overline F)^{\ell_i} 
\end{equation} for which this algebra is still $n$--CH.
For such an algebra we know, Theorem \ref{simc},  that there are some weights $a_{i,j}\in \mathbb N,\ i=1,\ldots,p,\ j=1,\dots,\ell_i$ associated to all the blocks $M_{k_i\cdot h_i}(\overline F)$  for which the trace is  given by a Formula analogous to formula  \eqref{trconp}.

We only need to show that, for each $i$  the $\ell_i$ weights $a_{i,j},\   j=1,\dots,\ell_i$ are equal. This follows from the fact that the Galois group  of  $\overline F$ over $F$ preserves the trace and permutes the $\ell_i$ summands  $M_{k_i\cdot h_i}(\overline F)$., Setting  $a_i:=a_{i,j}$ we have the result.                   \end{proof}
\begin{remark}\label{dimR}
Observe, from Formula \eqref{sspa}, that $\dim_FR\leq n^2$  and further, if  $\dim_FR= n^2$, then $R$ is a central simple $F$ algebra and  
$R\otimes_F\overline F= M_{n}(\overline F)$.\end{remark}

We can now generalize Theorem \ref{simc}.
\begin{theorem}\label{ssre1}
\strut  
\begin{enumerate}\item If    $S$ is a simple  trace algebra, and $tr(1)$ is invertible, then  its trace algebra is a field   $F$.
\item If  $S$ is an $n$--CH algebra with trace algebra  a field   $F$ and $K_S=0$,  then $S$ is finite dimensional over $F$, simple and isomorphic to one of the algebras  of Theorem \ref{trach}.
\item Assume  that $R$ is an $n$--CH algebra with trace algebra  a field   $F$, and $S:=R/K_R $  is   of rank $n^2$ over its center $F$. Then $K_R=0$ and $S$ is simple as algebra.\end{enumerate}

\end{theorem} 
\begin{proof}  1)\quad The fact that, if   $S$ is a simple  trace algebra, then the  trace algebra is a field   follows from Proposition \ref{tradi}.

2)\quad  This will follow, from Theorem \ref{simc},  if we can show that   the kernel  of $S\otimes_F\overline F$  is 0, where $\overline F$ is an algebraic closure of $F$.
 
 Now consider an element $\sum_{i=1}^k s_i\otimes f_i\in K_{S\otimes_F\overline F} $ with the $f_i\in \overline F$ linearly independent over $F$.
 
 If $a\in S$, then we have
 $$0=tr(a\sum_{i=1}^k s_i\otimes f_i)=\sum_{i=1}^k tr(as_i)\otimes f_i .$$ Since the $f_i\in \overline F$ are linearly independent over $F$ and we have $tr(as_i)\in F$,  this implies    that for all $a$ and for all $i$ we have $tr(as_i)=0$. So $s_i$ is in the kernel of $S$ which is assumed to be 0 hence $s_i=0$ and   $K_{S\otimes_F\overline F} =0$.

3)\quad As in part 2) we may reduce to the case $F$ algebraically closed  so that, from Remark \ref{dimR},  $R/K_R =M_n(F).$

 Then, since $R$ satisfies the PI of $n\times n$ matrices, we are under the hypotheses of Artin's Theorem \ref{ArT}, $R$ is a rank $n^2$ Azumaya algebra  over its center $A$.  Then,   by lifting idempotents, 
    $R=M_n(A)$ with $A$ a commutative algebra with Jacobson radical $J$ and $A/J=F$. 

   Finally, by   the main Theorem \ref{immu}, one has that $\mathrm j_R$ is an embedding and $A$ must be the trace algebra. So by hypothesis $A=F$.

 \end{proof}
\section{The Spectrum} 
\subsection{Prime CH algebras}
\begin{proposition}\label{mm} 1)\quad If $R$ is a semiprime $n$--CH algebra with trace algebra $A$  and $a\in A$ is not a zero divisor in $A$, then $a$ is not a zero divisor in $R$.
\smallskip

2)\quad  If $R$ is a  prime  algebra with trace, then the trace algebra $A$ is a domain and $R$ is torsion free relative to $A$.
\end{proposition}
\begin{proof}
1)\quad Let $J:=\{r\in R\mid ar=0\}$,  then $J$ is an ideal and we claim it is nil hence, by hypothesis, it is  0.  In fact 
taking trace  $0=t(ar)=at(r)$  implies $t(r)=0$  for all $r\in J$.   Thus the characteristic polynomial  of $r$ is $x^n$.  Since $r$ satisfies its characteristic polynomial,  we have $r^n=0$.

2)\quad          If $a\in A$ and $J:=\{r\in R\mid ar=0\}$, then both $J$  and $Ra\neq 0$  are trace ideals. We have $  RaJ=0$  from this the claim $J=0$. 
\end{proof}
\begin{theorem}\label{pos}
If $R$ is a  prime $n$--CH algebra with trace algebra $A$  and $G$ is the field of fractions of $A$, then  we have $R\otimes_AG$ is a simple $n$--CH algebra with trace algebra $G$ and  $R\subset R\otimes_AG$.
\end{theorem}
\begin{proof} Since $R$ is torsion free over $A$, we have $R\subset R\otimes_AG$.
Clearly the trace algebra of  $R\otimes_AG$ is $G$ so, by Theorem \ref{ssre1} it is enough to show that the kernel is 0. Since $G$ is the field of fractions of $A$,  each element $s$ of  $R\otimes_AG$ is of the form $s=r\otimes g$. If the element $s$ is in the kernel, then $r$ is in the  kernel of $R$, hence $r=0$. 
\end{proof} \begin{remark}\label{pspe}
As a consequence   $R\otimes_AG=\oplus_{i=1}^k      S_i$ with $S_i$ simple and finite dimensional over $G$. We see that, if $R$ is prime as $n$--CH algebra, then  the ideal $\{0\}=\bigcap _{i=1}^k     P_i$  is the intersection of the finitely many minimal prime ideals  $P_i=R\cap \oplus_{j=1,\  j\neq i}^k      S_j $. Finally $R/P_i\otimes_AG=S_i$.
\end{remark}\begin{remark}\label{pspe1}
This Theorem is stronger than Posner's Theorem  (cf. \cite{agpr} Theorem 2.4) which states that a prime PI $R$ ring has  a notn trivial centaer $Z$ with field of fractions $F$ and $R\otimes_ZF$ is a central simple $F$ algebra.
\end{remark}
\begin{corollary}\label{PIe}
A prime  $n$--CH  algebra $R$ is PI equivalent to one of the algebras $F(\underline m ;\underline a ) $ of Formula \eqref{Fma}.
\end{corollary}
\begin{proof}
Clearly $R$  is PI equivalent, with the notations of Theorem \ref{pos}, to its ring of fractions $R\otimes_TG$  which in turn is PI equivalent to $R\otimes_T\overline G$ with $\overline G$ the algebraic closure of $G$. Then  by Theorem \ref{simc}  $R\otimes_T\overline G$ is one of the algebras $\overline   G(\underline m ;\underline a ) $ which, as $F$ algebra, is  PI equivalent to   $F(\underline m ;\underline a ) $.
\end{proof}
\subsection{The spectrum as functor}
 In any associative algebra $R$  one can define the {\em spectrum}  of $R$ as the set of all its prime ideals, it is equipped with the {\em Zariski topology}.  
 
 For commutative algebras the spectrum is a contravariant functor; with $f:A\to B$ giving the map $P\mapsto f^{-1}(P)$. But in general a subalgebra of a prime algebra need not be prime  and the functoriality fails for non commutative algebras.

For algebras with trace $R$ we may define:
$$Spec_t(R):=\{P\mid P \text{ is a prime trace ideal}\}.$$  
If $T\subset R$  is the trace algebra of $R$, then, by Proposition \ref{mm} 2),  we have the map  $j: Spec_t(R)\to Spec (T),\ P\mapsto P\cap T.$   

For  an $n$--CH algebra $R$ we have the remarkable fact:
\begin{theorem}\label{iso}
The map  $j: Spec_t(R)\to Spec (T),\ P\mapsto P\cap T $ is a homeomorphism, its inverse  is $\mathfrak p\mapsto K(\mathfrak pR).$
\end{theorem}
\begin{proof} First, by  Proposition \ref{tradi} 1.   and any trace ideal $I\subset R$ we have that $I\cap T=t(I)=K(I)\cap T$. In fact if $a\in I\cap T$, then  we have $a=t(at(1)^{-1}) $ and, if  $a\in K(I)$,  we have $t(a)\in I$.  
                              
We first show that, for  an $n$--CH algebra $R$, the ideal $K(\mathfrak pR),$  is prime. 

This follows from  the characterization of prime algebras, Corollary \ref{CHC2} 2); since, by definition,  $t( K(\mathfrak pR )=t(\mathfrak pR )=\mathfrak p$. So,  $t(R/K(\mathfrak pR))=T/\mathfrak p$  is a domain and  the kernel of $R/K(\mathfrak pR)$ is $\{0\}$.
                
Hence the composition in one direction is the identity $\mathfrak p=K(\mathfrak pR)\cap T$.   

If $P$ is a prime ideal  we need to show that $P=K((P\cap T)R)$. We certainly have $P\supset K((P\cap T)R)$ so it is enough to show that, if $P\supset Q$ are two prime ideals and $P\cap T=Q\cap T$, then $P=Q$. In fact in $R/Q$ we have $t(P/Q)=0$ which implies $P/Q=0$.\end{proof}
\begin{corollary}\label{coff}
For morphisms of $n$--CH  algebras the spectrum is also a contravariant functor to topological spaces, setting $$f:A\to B,\quad  f^*:P\mapsto K(f^{-1}(P)).$$

\end{corollary} 
\subsection{Localization} \begin{proposition}\label{lallo} An  $n$--CH algebra $R$ is {\em local} (i.e. has a unique maximal trace ideal) if and only  of its trace algebra is local.

\end{proposition}
\begin{proof}
This is immediate from Theorem \ref{iso}.
\end{proof} Let $\mathfrak p$  be a prime ideal of $T$  and consider the local algebra $T_{ \mathfrak p}$ and $$R_{ \mathfrak p}:=R\otimes_TT_{ \mathfrak p}$$ we have then: 
\begin{corollary}\label{locanc}
$R_{ \mathfrak p}$ is a local ring with maximal ideal $K(R_{ \mathfrak p}\mathfrak p)$. 

Theorem \ref{trach} describes the possible residue and trace simple algebras $S_{ \mathfrak p}:=R_{ \mathfrak p}/K(R_{ \mathfrak p}\mathfrak p)$. 

\end{corollary}

One  also has, Remark \ref{pspe},  that there are only a finite number of prime (not trace invariant) ideals $P_i$  of $R$  with $P_i\cap T=\mathfrak p$. Moreover  in the next paragraph, Proposition \ref{azchh}, it will be shown  that $S_{ \mathfrak p}$ is the direct sum of the simple rings of fractions of $R/P_i.$

 This is a strong form of {\em going up} and {\em lying over} of commutative algebra in this general setting.\medskip

A special case is when $S_{ \mathfrak p}$ is simple as algebra and of rank $n^2$ over its center. In this case one can apply Artin's Theorem \ref{ArT} and deduce that
\begin{proposition}\label{locanc1}
  If $S_{ \mathfrak p}$ is simple as algebra and of rank $n^2$ over its center, then $R_{ \mathfrak p}$ is a rank $n^2$ Azumaya algebra over its center $T_{ \mathfrak p}$.

\end{proposition} \subsection{General prime  ideals}Let us analyze general prime  ideals, not necessarily closed under the trace. 

\begin{proposition}\label{azchh}
Let $S$ be  an $n$--CH algebra with trace algebra $A$, $P$ an algebra ideal of $S$  which is prime and $\mathfrak p=P\cap A$. Let $F$ be the field of fractions  of  $A/\mathfrak p$.

 Then $S/P\otimes_AF=S/P\otimes_{A/\mathfrak p}F$ is the ring of fractions of $S/P$, a  simple algebra, and $P$ is one of the minimal primes of the prime trace algebra $S/K(\mathfrak pS)$. \end{proposition}
\begin{proof}
 We have $P\supset K(\mathfrak pS) $ since $K(\mathfrak pS)$   is nilpotent modulo $\mathfrak pS$.  Thus we have a surjective map  $S/K(\mathfrak pS)\to S/P$  which induces a  surjective map  $S/K(\mathfrak pS)\otimes_AF\to S/P\otimes_AF$  with $F$ the field of fractions  of  $A/\mathfrak p$.   
 
 We have that $S/P$ is a prime algebra containing $A/\mathfrak p$ so it is torsion free over $A/\mathfrak p$ and $S/P\subset S/P\otimes_AF$ is contained in the ring of fractions  of $S/P$, Remark \ref{pspe1}.
 
 Since $S/P\otimes_AF=S/P\otimes_{A/\mathfrak p}F$ is a prime algebra and the simple trace algebra $S/K(\mathfrak pS)\otimes_AF=\oplus_iS_i$  is a direct sum of simple algebras, Remark \ref{pspe}, then   we must have   $S/P \otimes_AF=S_i$  for some index $i$ and the claim follows.
\end{proof}
\begin{corollary}\label{mama}
Let $S$ be  an $n$--CH algebra with trace algebra $A$, $M$ an algebra ideal of $S$  which is maximal and $\mathfrak m=M\cap A$. Then $\mathfrak m$ is a maximal ideal of $A$.
\end{corollary}
\begin{proof}
We have that $S/M$ is a simple algebra  integral over $A/\mathfrak m$ hence its center, a field, is  integral over $A/\mathfrak m$. The statement that $A/\mathfrak m$ is a field follows from the going up theorem of commutative algebra.\end{proof}
\begin{lemma}\label{azch}
Let $S$ be  an $n$--CH algebra with trace algebra $A$, $M$ an algebra ideal of $S$ so that $S/M$ is simple of dimension $n^2$ over its center  and $\mathfrak m=M\cap A$. Then $M=\mathfrak mS$.
\end{lemma}
\begin{proof}  Let    $\overline S:=S/\mathfrak m S$. $\overline S$  is an $n$--CH algebra with trace algebra $F=A/\mathfrak m  A $ which is  a field  by the previous Corollary.  The statement then follows  by part 3) of    Theorem  \ref{ssre1}.\end{proof}
\begin{theorem}\label{azch1}
Let $S$ be  an $n$--CH algebra with trace algebra $A$ a local ring with maximal ideal $\mathfrak m$. Let $M$ be an algebra ideal of $S$ so that $S/M$ is simple of dimension $n^2$ over its center, then $M= \mathfrak m S$ and $S$ is a rank $n^2$ Azumaya algebra over $A$.
\end{theorem}
\begin{proof}
The fact that $M= \mathfrak m S$ follows from the previous Corollary \ref{mama} since, by assumption,        $\mathfrak m$ is the unique maximal ideal of $A$.                   . 

The fact that $S$ is a rank $n^2$ Azumaya algebra over $A$ then follows from Artin's Theorem  \ref{ArT}, since, as any $n$--CH algebra, $S$ satisfies all polynomial identities of $n\times n$ matrices and no quotient satisfies the   polynomial identities of $n-1\times n-1$ matrices.
\end{proof}\vfill\eject

\part{The geometric theory}\section{The prime $T$--ideals}
 \subsection{Algebras of generic elements}
Given   a finite dimensional $n$--CH algebra $R$ over a field $F$ and with trace in $F$,  its relatively free algebra $F_R(\ell)$  in $\ell$  variables can be described by the {\em method of generic elements}  by fixing a basis  $u_1,\ldots,u_m$ of $R$  over $F$  and introducing $\ell m$ variables $\xi_{i,j}$  and generic elements 
$$\xi_i:=\sum_{j=1}^m\xi_{i,j}u_j,\ i=1,\ldots,\ell.$$ Then $F_R(\ell)$ is the subalgebra of $R\otimes_FF[\xi_{i,j}]$ generated by the generic elements $\xi_i$ and then closing it under trace.  The trace algebra $T_R(\ell)$  of $F_R(\ell)$  is contained in the polynomial ring $F[\xi_{i,j}]$, it is finitely generated over $F$  and the algebra $F_R(\ell)$  is a finitely generated torsion free module over $T_R(\ell)$ (Proposition 8.2.4  and Theorem 8.2.5 of \cite{agpr}).

Thus, if $G_R(\ell)$ is the field of fractions of  $T_R(\ell)$,  we may  construct the algebra  
 \begin{equation}\label{laccar}
H_R(\ell):=F_R(\ell)\otimes_{T_R(\ell)}G_R(\ell)\subset R\otimes_FF(\xi_{i,j})
\end{equation}  with trace in $G_R(\ell)$, hence finite dimensional over $G_R(\ell)$ by Proposition \ref{finl}.
 
We want to study  the relatively free algebra  in some $\ell$ variables for a simple algebra with trace as in Theorem \ref{trach}.  

Up to PI equivalence we can assume that the algebra $R$, with trace $t$, is   one of the algebras $F(\underline m ;\underline a ) $ of Formula \eqref{Fma}.  We describe such an algebra $R$ ordering the multiplicities $a_i$ as follows.
We have     $s$ integers $\ell_1<\ell_2<\cdots<\ell_s$ and $s$  partitions  $\underline k_i:=1^{h_{1,i}}2^{h_{2,i}}\cdots t_i^{h_{t_i,i}}$, so that $R$ is a   direct sum of the algebras  
 \begin{equation}\label{alss}
R=\oplus_{i=1}^sA_{\underline k_i},\ A_{\underline k_i}:= \oplus_j M_j(F)^{\oplus h_{j,i}},\ t(a)=\ell_i\  tr(a ),\ a\in M_j(F)^{\oplus h_{j,i}}.
\end{equation}
In this case a first general fact is
\begin{lemma}\label{dueg} Each algebra $R$ of Formula \ref{Fma} can be generated by 2 elements.

\end{lemma}
\begin{proof}
For a matrix algebra $M_n(F)$ a choice of two generators is a diagonal matrix $D$ with distinct entries  and the matrix $T_n$ of the cyclic permutation $(1,2,\ldots,n)$.   Just with $D$ we generate all diagonal matrices.  

So, if $R= \oplus_j M_{h_{j }}(F) $, let $D=\oplus_jD_j$  with all the entries distinct and $X=\oplus_jC_{h_{j }}$. $D$ and $X$ clearly   generate $R$.\end{proof}\begin{proposition}\label{nondeg} If $2\leq\ell\leq \infty$,  then  ($H_R(\ell)\stackrel{\eqref{laccar}}=F_R(\ell)\otimes_{T_R(\ell)}G_R(\ell)$):
\begin{align}
\dim_{G_R(\ell)}H_R(\ell)&=\dim_FR,\\ H_R(\ell)\otimes_{G_R(\ell)}F (\xi_{i,j})&=R\otimes_FF (\xi_{i,j}). \label{mmm}
\end{align}
The algebra $F_R(\ell)$ is prime and its ring of fractions $H_R(\ell)$ is a simple algebra with trace algebra $G_R(\ell)$.
\end{proposition}
\begin{proof}
A semisimple trace algebra $R$  is characterized by the fact that the trace form $tr(ab)$  is non degenerate.  If $m=\dim_FR$  we have that $m$ elements $u_1,\ldots,u_m\in R$ form a basis of $R$ if and only if the determinant of the $m\times m$ matrix  $tr(u_iu_j)$ is different from 0.

 Moreover, by the previous Lemma, $R$ can be generated by 2 elements. Therefore there are $m$  monomials $M_i,\ i=1,\ldots,m$ in  $\ell\geq 2 $ generic elements  such that the determinant $\Delta$ of the  matrix $tr(M_iM_j)$  is non zero.  If we invert $\Delta$  we then have that the $m$  monomials $M_i,\ i=1,\ldots,m$  form a basis of $F_R(\ell)[\Delta^{-1}]=F_R(\ell)\otimes_{T_R(\ell)}T_R(\ell)[\Delta^{-1}]$ over $ T_R(\ell)[\Delta^{-1}]$.  Hence they also form a basis of 
  $F_R(\ell)\otimes_{T_R(\ell)}F (\xi_{i,j})\subset R\otimes_FF (\xi_{i,j}).  $ These two spaces have the same dimension over $F (\xi_{i,j})$ so they coincide.  
  
  The final statement follows from  Formula \eqref{mmm}.
\end{proof} For $\ell=1$ the algebra $F_R(1)$ is commutative generated by a single generic element which is semisimple. Hence  $F_R(1)$ is also prime  and will be described later.

If $R$ is not semisimple, then  the dimension of $F_R(\ell)\otimes_{T_R(\ell)}G_R(\ell)$  may grow to infinity with $\ell$ as the simplest example shows. 

\begin{example}\label{dunu}
Consider  the algebra of dual numbers $F[\epsilon],\ \epsilon^2=0$ with trace of the multiplication, $tr(a+b\epsilon)=2a$. It is a 2 CH algebra.  

Denote the generic elements  $\xi_i=x_i+ y_i\epsilon$. The trace algebra is $F[x_1,\ldots,x_\ell], $ and $ G_R(m)=F(x_1,\ldots,x_\ell),$ we have
$$f(\xi_1,\ldots,\xi_\ell)= f(x_1,\ldots,x_\ell  )+(\sum_{j=1}^\ell\pd{f}{x_i}y_i)\epsilon.$$
$ H_R(\ell)=G_R(\ell) \oplus_{j=1}^\ell G_R(\ell)y_i \epsilon,\quad \dim_{G_R(\ell)}H_R(\ell)=\ell+1.$ 
\end{example}\medskip

This phenomenon is strictly  tied to the fact that $R$ is not semisimple.

\subsection{Classification of prime $T$--ideals}
We have now the following classification of prime $T$--ideals.
\begin{theorem}\label{prst0}  A $T$--ideal $I$  of $\mathcal S_n(\infty)$ is prime, if and only if, it is the ideal of trace identities of one of the algebras  $F(m_1,\ldots,m_k;a_1,\ldots,a_k) $.
\end{theorem}
\begin{proof} We have just seen that the $T$--ideal of one of these algebras is prime. Conversely if $I$ is prime it is the $T$--ideal of identities of the   prime algebra $\mathcal S_n(\infty)/I$.  By Corollary \ref{PIe}  a prime algebra is PI equivalent to one of the algebras   $F(m_1,\ldots,m_k;a_1,\ldots,a_k) $ hence  $I=I(m_1,\ldots,m_k;a_1,\ldots,a_k)  $.         \end{proof}
Let  us then set:
\begin{equation}\label{idueid} 
 J(m_1,\ldots,m_k;a_1,\ldots,a_k) :=I(m_1,\ldots,m_k;a_1,\ldots,a_k) \cap \mathcal C_n(\ell).
\end{equation}
By  Theorem \ref{iso}  we have:
\begin{equation}\label{IJ0}
I(\underline m ;\underline a ) =\{x\in \mathcal S_n(\ell)\mid tr(x\mathcal S_n(\ell))\subset J(\underline m ;\underline a ).
\end{equation}
The notion of $T$--ideal  can be defined also for  ideals in the trace algebra  $\mathcal C_n(\ell)$.   It is an ideal of  $\mathcal C_n(\ell)$ stable under all the trace compatible endomorphism of $\mathcal S_n(\ell)$ which, in turn, are obtained by substitutions of  the variables.

For the next result we need a straightforward generalization of 
Proposition 2.8. of \cite{agpr}.
This states that: 
\begin{proposition}\label{autinv}
An ideal of the free algebra, over an infinite field, in infinitely many variables is a $T$--ideal if and only if it is stable under all automorphisms of the free algebra.

\end{proposition}  \begin{lemma}\label{subid}
If  $J\subset  \mathcal C_n(\infty)$ is a $T$--ideal, then the set  \begin{equation}\label{IJ}
I(J):=\{x\in \mathcal S_n(\infty)\mid tr(x\cdot \mathcal S_n(\infty))\subset J\}
\end{equation} is    a $T$--ideal of $\mathcal S_n(\infty)$.

\end{lemma} 
\begin{proof}
Let $\phi:\mathcal S_n(\infty)\to \mathcal S_n(\infty)$ be a trace compatible automorphism and $x\in I(J)$. We have   that \[tr(\phi(x)\cdot \mathcal S_n(\infty))=   tr(\phi(x)\cdot \phi(\mathcal S_n(\infty))) =\phi( tr( x \cdot  \mathcal S_n(\infty)))  \subset J .\]
\end{proof} 
\begin{corollary}\label{ptt}
A $T$--ideal $J$  of $\mathcal C_n(\infty)$ is prime, if and only if, it is the ideal of pure trace identities of one of the algebras  $F(m_1,\ldots,m_k;a_1,\ldots,a_k) $.
\end{corollary} 
\begin{proof}
Let $I(J)$ be defined by Formula \eqref{IJ}. Since $I(J)$ is a $T$--ideal it is enough to show that $I(J)$ is also prime, by Theorem \ref{prst0}.

Let $A,B\supset I(J)$  be two trace ideals  so that $AB\subset I(J)$.

We thus  have $tr(A)tr(B)\subset  J$ and then, since $J$ is prime, we may assume that $tr(A)\subset  J$ so that $A\subset  I(J)$, which is thus prime.
\end{proof} 
{\bf Question} What can we say for  $\mathcal C_n(\ell),\ \ell<\infty$? If $\ell\geq n^2$ the same statement is true by Capelli's Theory, see page   \pageref{Capelli}.

\section{The relatively free algebra of a simple algebra}

 \subsection{Relatively free prime  algebras} For $R=  F(\underline m ;\underline a )  $ we have denoted by  $G_R(\ell)$ the field of fractions of   $T_R(\ell)$, the   trace algebra of its associated relatively free algebra $F_R(\ell)$. 
 
 The notations  are, cf. \eqref{laccar}:
\begin{equation}\label{notaf}
\begin{aligned}
F_R(\ell) =\mathcal S_n(\ell)/ I(\underline m ;\underline a ),  \quad   T_R(\ell)=\mathcal C_n(\ell)/J(\underline m ;\underline a ),&\\ \mathcal S_n(\ell)/ I(\underline m ;\underline a )\otimes_{\mathcal C_n(\ell)/J(\underline m ;\underline a )}G_R(\ell)=F_R(\ell)\otimes_{T_R(\ell)}G_R(\ell) = &H_R(\ell).
\end{aligned}
\end{equation}
The geometric interpretation, of these ideals of the ring of invariants of matrices, will be developed in \S \ref{geoti}, Theorem \ref{tit} and Corollary \ref{strf}.                   

The next Theorem is a generalization of Theorem \ref{teoAmi} of Amitsur.

\begin{theorem}\label{prst} Let $R$ be as in Formula \ref{alss} and $\ell\geq 2$ then
\begin{equation}\label{iD0}
\mathcal S_n(\ell)/ I(\underline m ;\underline a )\otimes_{\mathcal C_n(\ell)/J(\underline m ;\underline a )}G_R(\ell)=H_R(\ell)\simeq  \oplus_{i=1}^s  \oplus_j D_ {  h_{j,i}}.
\end{equation}   $D_ {  h_{j,i}}$ a division algebra of dimension $j^2$ over its center $F_{j,i}$ which contains  $G_R(\ell)$.  It has degree $[F_{j,i}:G_R(\ell)]=h_{j,i}$. The trace is given by a Formula as in \ref{trconp}.
\end{theorem}
\begin{proof}
Given any pair $h,k$ of integers  there is a division ring $L_{h,k}$ with center $G_{h,k}$  and $\dim_{G_{h,k}}L_{h,k}=h^2$  furthermore one can choose $G_{h,k}$ as to contain a subfield $F_{h,k}$ with $[G_{h,k}:F_{h,k}]=k$  and all these contain $F$. 

 Then $L_{h,k}$ equipped with the trace  $tr_{L_{h,k}/F_{h,k}}= tr_{G_{h,k}/F_{h,k}}\circ tr_{L_{h,k}/G_{h,k}}$ is PI equivalent to $M_h(F)^{\oplus k}$. Thus,  to compute the relatively free algebra we may replace $R$ by $S:=\oplus_{i=1}^s  \oplus_j L_ { j, h_{j,i}}$  with the appropriate trace and $F_R(\ell)\simeq F_S(\ell)$.  The same argument as in Proposition \ref{nondeg}  gives, using a basis of $S$ for the generic elements
$$F_S(\ell)\otimes_{T_R(\ell)}F (\xi_{i,j})\simeq  S\otimes_FF (\xi_{i,j})=\oplus_{i=1}^s  \oplus_j L_ { j, h_{j,i}}\otimes_FF (\xi_{i,j}).  $$ Each $L_ { j, h_{j,i}}\otimes_FF (\xi_{i,j})$ is still a division algebra and the claim follows.  \end{proof}
\subsection{An explicit example}
Let us develop the simplest example for the    algebra $R=F^{\oplus s}$. Take as trace $tr(x_1,\ldots,x_s)=\sum_{i=1}^sa_ix_i$  for $s$ integers $a_1\leq a_2\leq\ldots\leq  a_s$  with $\sum_i a_i=n$. Then $R$ can be thought of as the subalgebra of diagonal $n\times n$ matrices in which each entry $x_i$  appears with multiplicity $a_i$. It is an $n$--CH algebra.  

The  corresponding free  algebra,  for $\ell=1$ is described as follows. Consider $s$ variables $x_1,\ldots,x_s$  over $\mathbb Q$  (or $\mathbb C$) and let $G:=\mathbb Q(x_1,\ldots,x_s)$ be the field they generate. Let $X$ be a diagonal $n\times n$ matrix with entries $x_i$  a number $ a_i$ of times.  The relatively free algebra of $R$ in one variable is   the trace algebra generated over $\mathbb Q$ by $X$.\smallskip

\begin{example}\label{sica}
First the simplest case $s=2, a_1=1, a_2=2,n=3,\ x_1=u,x_2=x_3=v$.
\end{example}
 The characteristic polynomial of $X$ is: 
$$(t-u)(t-v)^2=t^3-(u+2v)t^2+( 2uv+v^2)t-uv^2.$$
Set $a=(u+2v);\ b=( 2uv+v^2);\ c=uv^2$.  These 3 elements  generate the algebra $A$ of traces.

We have then:
\begin{align}
a^2-3b =(u-v)^2;\ ab-9c=2v(u-v)^2;\ &9c+a^3-4ab = u(u-v)^2,\\\implies  \ 2v=\frac{ab-9c}{a^2-3b}
;\ u&=\frac{9c+a^3-4ab}{a^2-3b}\label{cled}\end{align} 
Moreover
$$3v^2-2av+b=0;\ u^2-4au+a^2-4b=0,\ $$ and $a,b,c$ are the restrictions of the elementary symmetric functions to one of the 3 planes where two coordinates are equal, therefore they satisfy as equation the {\em discriminant} a polynomial of degree 4 and weight 6:
$$3a^2b-162abc+243c^2-12a^2b^2+18a^3c+36b^3 .$$
The element $X$, besides satisfying its characteristic polynomial of degree 3, satisfies its minimal polynomial $(t-u)(t-v)$ which can be made into a polynomial with coefficients in $A$  by multiplying it by $(a^2-3b)^2$, Formula \eqref{cled}. 
 \begin{equation}\label{lal}
((a^2-3b)X^2-(9c+a^3-4ab))((a^2-3b)X-\frac12(ab-9c)) ,
\end{equation}The relative free algebra with trace of  $R$, in 1  variable,  is the algebra generated by $X,a,b,c$  modulo these two relations.

The relation \eqref{lal} decomposes as a product of two factors hence the ring of fractions of $R$ is the direct sum of  two fields isomorphic to $\mathbb Q(u,v)$ and under this isomorphism $X\mapsto (u,v)$ the trace on $\mathbb Q(u,v)\oplus  \mathbb Q(u,v)$ is  $(\alpha,\beta)\mapsto\alpha+2\beta$.
\medskip
\subsection{The general one variable case}
We return to the general case of $X$   a diagonal $n\times n$ matrix, with entries  $p$ variables $x_i$, each repeated a number $ a_i>0$ of times. \smallskip

Let $\mathbb Q[X]_T:=\mathbb Q[X, tr(X^i)],\ i=1,\ldots n$ the algebra with trace generated by $X$ and $T(X)= \mathbb Q[  tr(X^i)]$ its trace algebra,  $X$ is integral over $T(X)$ since it satisfies its characteristic polynomial with coeffients in $T(X)$.

Let    $F\subset G=\mathbb Q(x_1,\ldots,x_p)$ be the field  of fractions of $T(X)$. We have that $X$ is algebraic over $F$  and $F[X]\subset G^p$ is semisimple. \smallskip

The set of the $p$  variables $x_i$ is  partitioned, for some $s$,  into the $s$ equivalence classes  of the equivalence relation $x_i\cong x_j,\ \iff  a_i= a_j$. Each class  $A_i$   with $k_i$ elements.
 
In other words,    we have  strictly positive integers  $k_1,k_2,\ldots,k_s$ so that  $\sum_ik_i=p$, and also $0< a_1< a_2<\ldots < a_s$ so that $n=\sum k_i a_i$, and  $A_i$ has $k_i$ elements.  The algebra  $T(X)$ is also generated by the coefficients $\alpha_j$ of the characteristic polynomial of $X$:
\begin{equation}\label{nJ}
h(t; x_1,\ldots,x_p )=\prod_{i=1}^s \prod_{x_j\in A_i} (t-x_j) ^{ a_i}.
\end{equation}  

In the special case in which all $ a_i=1$ we have that $T(X)$ is the algebra of symmetric  polynomials  $\mathbb Q[x_1,\ldots,x_n]^{S_n}=\mathbb Q[e_1,\ldots, e_n]$ with $e_i$ the elementary symmetric functions 
$$h(t; x_1,\ldots,x_n ):\prod_{i=1}^n(t-x_i) =t^n+\sum_{j=1}^{n}(-1)^je_j   t^{n-j}$$ and $T(X)[X]=\mathbb Q[e_1,\ldots, e_n][t]/(t^n+\sum_{j=1}^{n}(-1)^je_j   t^{n-j})\simeq \mathbb Q[e_1,\ldots, e_n][x_1].$\smallskip

  As for $F$,   the field  of fractions of $T(X)$  is also  generated by the coefficients $\alpha_j$ of the characteristic polynomial of $X$,. We claim that $F$  contains $ \mathbb Q(x_1,\ldots,x_p)^{S_p}$ so, 
   by Galois Theory  we have $F=\mathbb Q(x_1,\ldots,x_p)^H$ where $H$ is the subgroup of $S_p$  fixing the polynomial $h(t; x_1,\ldots,x_p )$.   The group $H$ is clearly the Young subgroup $\prod_{i=1}^sS_{A_j}$  product of the symmetric groups on the disjoint sets $A_j$.\smallskip

    In order to show that $F$ contains $ \mathbb Q(x_1,\ldots,x_p)^{S_p}$ we need a simple Lemma:
\begin{lemma}\label{irrf} Lat $L$  be a field of characteristic 0 and $f(t)\in L[t]$  a polynomial whose distinct roots are $y_1,\ldots,y_k$  (they may appear with multiplicities).

Then the polynomial  $\bar f(t):=\prod_{i=1}^k(t-y_i)$  has coefficients in $L$.
\end{lemma}
\begin{proof}
Decompose $f(t)=\prod_{i=1}^jg_i(t)^{h_i}$  with the $g_i(t)\in L[t]$ irreducible and distinct. Then, since $L$ has characteristic 0, all the $g_i(t)$  have distinct roots and clearly  $\bar f(t)=\prod_{i=1}^jg_i(t).$
\end{proof}

 Applying this Lemma to the polynomial $h$ of Formula \eqref{nJ} we have that 
$$\mathbb Q(x_1,\ldots,x_p)^{S_p}\subset  F\subset \mathbb Q(x_1,\ldots,x_p).$$

\begin{proposition}\label{lF}
\begin{align}\label{lF1}
F[X]&=\oplus_{i=1}^sF[t]/\prod_{x_j\in A_i} (t-x_j)=\oplus_{i=1}^sG_i,\quad G_i\  \text{a field}.
\end{align}
The element $X$ corresponds to $(z_1,\ldots,z_s)$ with $z_i$ the class of $t$  in the $i^{th}$ summand $G_i$, a field of degree $k_i$ over $F$.

The trace is
$$tr(r_1,\ldots,r_s)=  \sum_{i=1}^s a_i\, tr_{G_i\backslash F}(r_i).$$
\end{proposition}\begin{proof} The polynomial $g_i(x):=\prod_{x_j\in A_i} (t-x_j)$   has coefficients in $F$. since it is $H$ invariant,  and 
    it is  irreducible over $F$, since the elements $x_j,\ x_j\in A_i$  form a unique orbit under the group $H$.

\end{proof}
Remark that $G_i\simeq F[x_i],\ x_i\in A_i.$

As for the trace algebra   $T(X)= \mathbb Q[ tr(X^i)]$, this is also  generated by the coefficients of the characteristic polynomial of $X$, Formula \eqref{nJ}. 

The matrix $X$ is the direct sum $X=\oplus_{i=1}^sX_i^{\oplus a_i}$ of the diagonal matrices $X_i  $ of size $k_i$ having as entries the variables  $ x_j,\ x_j\in A_i$. We have that the characteristic polynomial of $X$ is $\prod_{i=1}^s g_i(x)^{a_i}$, with $g_i(x)$ its irreducible factors, over $F$. 

 If we let $T_i(X)= T(X_i)$ we have that $T(X_i)$ is the polynomial ring in the $k_i$ elementary symmetric functions in these variables and finally:
\begin{theorem}\label{inte}  1)\quad The integral closure $\overline T(X)$ of $T(X)$ in its field of fractions $F$ is
\begin{equation} 
\overline T(X) =T(X_1)\otimes_{\mathbb Q}T(X_2)\otimes_{\mathbb Q}\ldots\otimes_{\mathbb Q} T(X_{s-1}) \otimes_{\mathbb Q} T(X_{s})=\mathbb Q[x_1,\ldots,x_p]^H. 
\end{equation}
2)\quad The integral closure of $T(X)[X]$ in the semisimple algebra $F[X]$ is
\begin{equation} 
\oplus_{i=1}^s\overline T(X) [t]/\prod_{x_j\in A_i} (t-x_j).
\end{equation}
    
\end{theorem}
\begin{proof}1)\quad $\overline T(X) $ is a polynomial algebra so it is  integrally closed.  The elements $x_i$ are integral  over  $T(X)$ so that    $\overline T(X) $    is integral  over  $T(X)$. Finally, from Lemma \ref{irrf},  $\overline T(X) $    and  $T(X)$ have the same field of fractions $F$.

2)\quad The integral closure of $\overline T(X) $ in  $\oplus_{i=1}^s G_i$ is the direct sum of the integral closures of the projections of $\overline T(X) $ in the factors $G_i=F[t]/g_i(t)$, which are clearly  $\overline T(X)\otimes _{T(X_i)}T(X_i)[t]/g_i(t)$.

\end{proof}
\medskip

\section{The general case $\ell\geq 2$\label{geoti}}  
Contrary to the algebra $\mathcal S_n(\ell)$ the algebras $\mathcal S_n(\ell)/ I(\underline m ;\underline a )$ do not have a clear description by invariant theory, in particular we will make a conjecture on its universal embedding into $n\times n$ matrices, see page \pageref{conj1}.

We want to present a form of {\em integral closure, $S_\ell(\underline m ;\underline a )$,} of  $\mathcal S_n(\ell)/ I(\underline m ;\underline a )$ which has a clear description by invariant theory, \ref{seee}.

\subsection{The Luna stratification I}

We assume $F$ algebraically closed.  Let us first recall some general facts, of Geometric Invariant Theory, for which we may refer to Springer \cite{spr}. 

Let $G$ be a reductive group and $V$ a linear representation. One has that the ring of invariants $S[V^*]^G$   is finitely generated, and  so it is the coordinate ring of some irreducible variety denoted by $V//G$. 

The inclusion of algebras    $S[V^*]^G\subset  S[V^*] $ corresponds to  a {\em quotient map } $\pi:V\to V//G$  which is surjective and such that each fiber contains a unique closed orbit. The stabilizer of a point of a closed orbit is also reductive, one says that two stabilizers have the same {\em type} if they are conjugate. One has only finitely many types of stabilizers.    
 By Luna's theory, \cite{Luna},  the quotient variety is stratified into the smooth strata  $\Sigma_H$ formed by all points in $ V//G$ corresponding to the stabilizer types $H$. That is the closed orbit in $\pi^{-1}(p)$ is $G$--isomorphic to $G/H$ or, equivalently, there is a point $q\in   \pi^{-1}(p)$ with stabilizer $H$.  Finally the closure of a stratum  $\Sigma_H$ is the union of the strata  $\Sigma_K$ where $H$ is conjugate to a subgroup of $K$.\smallskip
 
 Our case is  the space $M_n(F)^{\ell}$,   acted by the group $PGL(n,F)$  of automorphisms of $M_n(F) $. The action is  induced by conjugation by the linear  group $GL(n,F)$. 
 
By the Theorem of M. Artin \ref{ArT}, \cite{ArM} in this case,  the non $\{0\}$ closed orbits  correspond to  those $\ell$--tuples of matrices which generate a semisimple subalgebra $C$ of $M_n(F) $.  Now, any such subalgebra $C$ is  conjugate, under the automorphism group $PGL(n,F)$ of $M_n(F) $,   to one of the algebras   of Formula \eqref{Fma} $F(\underline m ;\underline a )=F(m_1,\ldots,m_k;a_1,\ldots,a_k),\ \sum m_ia_i=n$. We assume from now on that $\ell\geq 2$ so, from Lemma \ref{dueg},  all of these algebras appear.

\begin{proposition}\label{dct}

The stabilizer  of a   $\ell$--tuple  of matrices, generating the algebra  $F(\underline m ;\underline a)$, is the group $G(\underline a ;\underline m)$ of invertible elements in the centralizer algebra,  conjugate to  $F(\underline a ;\underline m)$, and of dimension $\sum_ia_i^2$.

\end{proposition}
\begin{proof}
In fact an element $g\in GL(n,F)$  fixes the $\ell$ tuple if and only if it lies in the centralizer of the algebra  $F( \underline m ;\underline a )$ which is an algebra of type $F(\underline a ;\underline m )$ by the double centralizer Theorem.
\end{proof}
 \begin{definition}\label{ost}
A closed orbit $\mathcal O\subset M_n(F)^{\ell}$, under the group $PGL(n,F)$, with stabilizer type $G(\underline a ;\underline m)$ will be called of {\em type} $ (\underline a ;\underline m)$.
\end{definition}
Note that an, orbit of type  $ (a_1,\ldots,a_k;m_1,\ldots,m_k)$
 has dimension equal to  $n^2-\sum_{i=1}^ka_i^2$.
 
\begin{corollary}\label{sttyp}
In the case of $M_n(F)^\ell$ we have  that the stabilizer types are of the form $G(\underline a ;\underline m)$.   

 We denote by $\Sigma^\sharp (\underline a ;\underline m)=\Sigma^\sharp (a_1,\ldots,a_k;m_1,\ldots,m_k)$    the  union of the closed orbits of type $ (\underline a ;\underline m)$.  By $\Sigma  (\underline a ;\underline m)=\Sigma(a_1,\ldots,a_k;m_1,\ldots,m_k)$ the corresponding stratum in the quotient variety.
\end{corollary} 

The  variety $\Sigma^\sharp  (\underline a ;\underline m)$ can be described as follows.    
\begin{lemma}\label{int} A closed orbit $\mathcal O\subset M_n(F)^{\ell}$, under the group $PGL(n,F)$, with stabilizer type $G(\underline a ;\underline m)$, intersects the space        $F(\underline m ;\underline a )^\ell$ in a single orbit under the group $Aut(\underline m ;\underline a )$ of automorphisms  of $F(\underline m ;\underline a )$.

\end{lemma}
\begin{proof} By definition there is a point $p:=(p_1,\ldots,p_\ell)$ in the orbit with stabilizer exactly  $G(\underline a ;\underline m)$. Thus the elements $\{p_1,\ldots,p_\ell 
\}$ lie in $F(\underline m ;\underline a )$ the centralizer of $G(\underline a ;\underline m)$. Moreover these elements must generate $F(\underline m ;\underline a )$, again by the double centralizer Theorem,  since they generate a semisimple algebra with centralizer  $F(\underline a ;\underline m)$.
Assume that $p,q\in  F(\underline m ;\underline a )^\ell$ are in the same orbit $q=g\cdot p$. Since $p$ and $q$  generate $F(\underline m ;\underline a )$ we must have that $g$ fixes $F(\underline m ;\underline a )$  and of course it is an automorphism of  $F(\underline m ;\underline a )$.\smallskip

Conversely it is enough to show that:
\begin{lemma}\label{eaai}
Each automorphism of $F(\underline m ;\underline a )$ is the restriction of an  automorphism of matrices $M_n(F)$.
\end{lemma} 
\begin{proof}
The automorphism group
 \begin{equation}\label{datu}
Aut(\underline m ;\underline a )=G_0\ltimes H,\ \dim Aut(\underline m ;\underline a )=\sum_im_i^2-k
\end{equation}  is the semidirect product of its conneced component $G_0=\prod_{i=1}^kPGL(m_i,F)$ of inner automorphisms and the finite group $H,$ product of the symmetric groups permuting the blocks relative to indices $(m_i,a_i)$  which are equal between each other. These are also given by conjugation by permutation matrices.
\end{proof}
 \end{proof} 
 
\begin{proposition}\label{dims} With $ (\underline a ;\underline m)= (a_1,\ldots,a_k;m_1,\ldots,m_k).$
\begin{equation}\label{dims0}
\dim \Sigma^\sharp   (\underline a ;\underline m)=n^2+ (\ell-1) ( \sum_{i=1}^km_i^2)- \sum_{i=1}^ka_i^2 +k.
\end{equation}

\end{proposition}
\begin{proof} For each $p\in \Sigma^\sharp   (\underline a ;\underline m)$ denote by $F(p)$ the semisimple subalgebra it generates, by hypothesis conjugate to $F(\underline m ;\underline a )$.  

The set of subalgebras of $M_n(F)$ conjugate to $F(\underline m ;\underline a )$ is an orbit,  in the Grassmann variety  of subspaces of  $M_n(F)$, under $GL(n,F)$. The subgroup   $$\tilde G(\underline m ;\underline a ):=\{g\in GL(n,F)\mid g(F(\underline m ;\underline a ))=F(\underline m ;\underline a )\}   $$  fits into an exact sequence (Lemma \ref{eaai})
$$ \begin{CD}
0@>>> G(\underline a ;\underline m )@>>>\tilde G(\underline m ;\underline a )@>>> Aut(\underline m ;\underline a )@>>>0
\end{CD}.$$  Thus this orbit of $F(\underline m ;\underline a )$  has  dimension  $$\dim GL(n,F)-\dim \tilde G(\underline a ;\underline m )\stackrel{\eqref{datu}}=n^2-     \sum_ia_i^2-\sum_im_i^2+k .   $$ The variety   $ \Sigma^\sharp    (\underline a ;\underline m)$ fibers, in a  $GL(n,F)$ equivariant way over this orbit.  
           
The  fiber of $ F(\underline m ;\underline a )$ is the set of points $\mathcal O(\underline m ;\underline a ,\ell):=\{p\in \Sigma^\sharp    (\underline a ;\underline m)\}$ with $F(p) =F(\underline m ;\underline a )$. 

Thus  
  $  \mathcal O(\underline m ;\underline a ,\ell) $  is the open set of $F(\underline m ;\underline a )^\ell$ of $\ell$--tuples which generate $F(\underline m ;\underline a )$.  We have  $\dim\mathcal O(\underline m ;\underline a ,\ell)=\ell \sum_im_i^2, $ hence the claim.
  \end{proof}
\subsection{The Luna stratification II}

We now pass to the Luna strata $ \Sigma   (\underline a ;\underline m)$ in the quotient variety.

\begin{proposition}\label{dims1} With $ (\underline a ;\underline m)= (a_1,\ldots,a_k;m_1,\ldots,m_k).$
\begin{equation}\label{dims2}
\dim \Sigma   (\underline a ;\underline m)= (\ell-1)( \sum_{i=1}^km_i^2)  + k.
\end{equation}

\end{proposition}
\begin{proof} This follows from the previous Proposition and the fact that the closed   orbits of type $(\underline a ;\underline m)$ have dimension $n^2-\sum_ia_i^2.$
 \end{proof} \begin{remark}\label{destra}Recall that, by Luna's Theorem, a stratum  $ \Sigma   (\underline a' ;\underline m')$  is contained in the closure of another stratum  $ \Sigma   (\underline a ;\underline m)$ if and only if  the stabilizer $G(\underline a' ;\underline m')$ contains  a conjugate of the  stabilizer $G(\underline a ;\underline m)$.  Equivalently this means that the algebra $F(\underline m' ;\underline a' )$ is contained in   a conjugate of the   of the algebra $F(\underline m ;\underline a )$.

\end{remark}\smallskip

 Consider   a maximal  semisimple proper subalgebra  $S\simeq F(\underline m' ;\underline a' )$ of a  semisimple  algebra $ R\simeq F(\underline m  ;\underline a  )$ over $F$ algebraically closed. 
 \begin{lemma}\label{maxs}
We have one of the following possibilities:

1)\quad If $R$  is simple  there is a decomposition of  $m =p+q$, so that the pair $S\subset R$ is isomorphic to $M_{p}(F)\oplus M_{q}(F)\subset M_{m }(F)$.

2)\quad There is a decomposition  $R=R_1\oplus R_2,\ S=S_1\oplus R_2$ with $R_1$ simple  and $S_1$ is maximal in $R_1$.

3)\quad There is a decomposition  $R=R_1\oplus R_2\oplus R_3,\ S=S_1\oplus R_3$ and $S_1\subset R_1\oplus R_2$ with an isomorphism $S_1\simeq R_1\simeq  R_2\simeq  M_m(F)$ and the inclusion  $S_1\subset R_1\oplus R_2$ is the diagonal of $M_m(F)\oplus M_m(F)$.
\end{lemma}
\begin{proof}
1) follows from the classification of semisimple subalgebras of    $M_m(F)$, Proposition \ref{ssn}.

2) and 3) are by induction on the number $k$ of simple summands of $R$. If $k=1$ we are in the previous case. So assume $R=R_1\oplus R_2$  with $R_1$ simple. let $\pi_i:S\to R_i, i=1,2$ be the two projections> If one of the two projections is not  surjective, say $\pi_i(S)=P\subsetneq R_i$, then $S=\pi_i^{-1} R_i$ and we apply induction.

Assume now $\pi_1,\pi_2$ surjective and let $K:=ker(\pi_1)$  then $S\simeq j(R_1)\oplus K,\ K\subset R_2$ with $j$ some embedding of $R_1$ in $R$. We claim that $\pi_2$ is an isomorphism, otherwise its kernel must be  $ j(R_1)$ and then $K=R_2$ and $S=R$. Thus $R=R_1\oplus j(R_1)\oplus K$ and $S$ is the diagonal of $R_1\oplus j(R_1)$ direct sum $K$.

\end{proof}
  The dimension of the corresponding strata $\Sigma   (\underline a' ;\underline m')$  is in case 2)
 $$\dim \Sigma   (\underline a ;\underline m)-(\ell-1)(p+q)^2+(\ell-1)(p^2+q^2)+1=\dim \Sigma   (\underline a ;\underline m)-(\ell-1) 2pq+1. $$
We have $(\ell-1) 2pq-1=1$ if and only if  $\ell=2,\ p=q=1$. 

In case 3) it is
$$ (\ell-1)(2m^2+ \sum_{i=3}^km_i^2)  + k-( (\ell-1)(m^2 +\sum_{i=3}^km_i^2)  + k-1)= (\ell-1) m^2+1
$$
which, for $\ell \geq 2$ is never equal to 1.

So we have:
\begin{theorem}\label{cod}   The complement of the stratum $\Sigma   (\underline a ;\underline m)$ in its closure $\overline\Sigma   (\underline a ;\underline m)$ has codimension $\geq 2$; except in the case $\ell=2$ and one of the blocks is $2\times 2$  matrices.

In the case  case of codimension $\geq 2$ the regular functions on $\Sigma   (\underline a ;\underline m)$  give the integral closure of  the functions on $\overline\Sigma   (\underline a ;\underline m)$.

\end{theorem}
For two $2\times 2$ matrices $X,Y$ the complement of the open stratum is the hypersurface of equation $\det(XY-YX).$\medskip

Recall that a $T$--ideal in the free algebra  with trace $F \langle x_1,\ldots,x_\ell \rangle[t(M)]$ is an ideal stable under the semigroup $\mathcal S$ of {\em variable substitutions} or  {endomorphisms}.

 {Such a semigroup induces for all $n$    }
\begin{enumerate}

\item a semigroup $\mathcal S_n$ of   variable substitutions or  endomorphisms in each of the quotient algebras $F_{T,n}\langle x_1,\ldots,x_\ell \rangle$

\item an opposite semigroup $\mathcal S_n^{op}$ of regular maps on the variety $ M_n(F)^{\ell}$ commuting with the $PGL(n,F)$ action

\item a  semigroup of regular maps in the quotient variety $ M_n(F)^{\ell}//PGL(n,F)$ or  on its coordinate ring $\mathcal C_n(\ell)$ the trace algebra of  
 $F_{T,n}\langle x_1,\ldots,x_\ell \rangle$.

\end{enumerate}

     Let  $\mathcal G$  be the (infinite dimensional) group of automorphisms of the free algebra with trace $R$ in the variables $X=\{x_1,\ldots,x_\ell\}$.  Since the  algebra $\mathcal S_n(\ell)$  is relatively free,Theorem \ref{SFT},  then  $\mathcal G$  induces, for each $n$,  a (still infinite dimensional)  group of automorphisms of $\mathcal S_n(\ell)$.  We denote this group by  $\mathcal G_n(\ell)$.\smallskip

     Then, by the universal construction for $ \mathcal S_n(\ell)$, Proposition \ref{genma} and Definition \ref{genmadt}, the group $\mathcal G_n(\ell)$ induces a group of automorphisms on the commutative algebra $ F[\xi_{h,k}^{(i)}]$ which, by Proposition \ref{inva},  commutes with  $PGL(n,F)$.  
     
 \begin{proposition}\label{geaut}
The group $\mathcal G_n(\ell)$ induces a group of regular automorphisms on the variety  $M_n(F)^\ell $ commuting with  $PGL(n,F)$.

Therefore it induces a group of regular automorphisms on the quotient variety  $M_n(F)^\ell //GL(n,F)$ which preserves the  Luna stratification.
\end{proposition} 
\begin{proof}
The only point to be proved is that it preserves the  Luna stratification.
This is clear  since an  automorphism  $g$ of  $M_n(F)^\ell $, commuting with  $PGL(n,F)$, maps closed orbits into closed orbits and 
for any  $p\in  M_n(F)^\ell $ the stabilizers, in $PGL(n,F)$ of $p$ and $g\cdot p$ coincide.\end{proof}   
{\bf Conjecture 1}  The Luna strata  of  $M_n(F)^\ell //GL(n,F)$  are the orbits of  the group $\mathcal G_n(\ell)$.  Observe that, by  Proposition \ref{autinv}, this is true in a {\em stable} form.\smallskip
          
    \subsection{$\mathcal G_n(\ell)$ stable ideals}A  $\mathcal G_n(\ell)$ stable ideal $I$ of $M_n(A_{\ell,n})^{PGL(n,F)}$    defines  a $PGL(n,F)$ stable subscheme of $M_n(F)^\ell $ stable also under $\mathcal G_n(\ell)$.
    
    Before proving the Theorem of this section we need a Lemma.
    \begin{lemma}\label{unle}
There exists an element $a\in   T_R(\ell)$ such that:
 \begin{equation}\label{iD2}
\mathcal S_n(\ell)/ I(\underline m ;\underline a )[a^{-1}]=\mathcal S_n(\ell)/ I(\underline m ;\underline a )\otimes_{T_R(\ell) }T_R(\ell)[a^{-1}] \simeq  \oplus_{i=1}^s  \oplus_j A_ {  h_{j,i}}
\end{equation}
with the following properties: 

1)\quad   $A_ {  h_{j,i}}$ is a  domain and an Azumaya algebra free  of dimension $j^2$ over its center $Z_j$. 

2)\quad $Z_{  h_{j,i}}$ is free over  $T_R(\ell)[a^{-1}]$, has degree $[Z_{  h_{j,i}}:T_R(\ell)[a^{-1}] ]=h_{j,i}$ and is unramified. 
\end{lemma}
\begin{proof}  First, by inverting some element $c\in T_R(\ell)$, clearing a finite number of denominators,  we may decompose
\begin{equation}\label{iD4}
\mathcal S_n(\ell)/ I(\underline m ;\underline a )[a^{-1}]=\mathcal S_n(\ell)/ I(\underline m ;\underline a )\otimes_{T_R(\ell) }T_R(\ell)[a^{-1}] \simeq  \oplus_{i=1}^s  \oplus_j B_ {  h_{j,i}}
\end{equation} with $D_ {  h_{j,i}}=B_ {  h_{j,i}}\otimes_{T_R(\ell)[c^{-1} }G_R(\ell)=D_ {  h_{j,i}}$. We  may further assume that  $B_ {  h_{j,i}}$ is free of rank $n^2$  over its center $\bar Z_ {  h_{j,i}}$ with $F _ {  h_{j,i}}$ as field of fractions.

Choose a  field extension $K_R(\ell)$ of  $G_R(\ell)$ of some finite degree $N$ and 
 \begin{equation}\label{iDx}
\mathcal S_n(\ell)/ I(\underline m ;\underline a )\otimes_{T_R(\ell) }K_R(\ell)=H_R(\ell)\simeq  \oplus_{i=1}^s  \oplus_j M_j(K_R(\ell))^{ h_{j,i}}
\end{equation} with  $D_ {  h_{j,i}}\otimes_{G_R(\ell)} K_R(\ell)=M_j(K_R(\ell))^{ h_{j,i}} $ for each $i,j$. Then we may choose    an element  $b\in K_R(\ell)$ so that $  K_R(\ell)=G_R(\ell)[b]$ and  the discriminant $d$ of  the basis  $1,b,\dots, b^{N-1}$  of  $K_R(\ell)$ over   $G_R(\ell)$ is in $T_R(\ell)$. Furthermore    there is some non zero element    $f\in T_R(\ell)$   so that, setting $a:=c\cdot d\cdot f$ we have
\begin{equation}\label{iD3}
\mathcal S_n(\ell)/ I(\underline m ;\underline a )[a^{-1}][b]  \simeq  \oplus_{i=1}^s  \oplus_j M_j(T_R(\ell)[a^{-1}][b])^{ h_{j,i}} 
\end{equation}
That is  all the  matrix units $e_{i,j}$ of all summands are in   $\mathcal S_n(\ell)/ I(\underline m ;\underline a )[a^{-1}][b] $ and 
$$ B_ {  h_{j,i}}\otimes_{T_R(\ell)[c^{-1}}T_R(\ell)[a^{-1}][b]= M_j(T_R(\ell)[a^{-1}][b])^{ h_{j,i}}.$$  
Set $A_ {  h_{j,i}}:=   B_ {  h_{j,i}}\otimes_{T_R(\ell)[c^{-1}]}T_R(\ell)[a^{-1}]$  we claim that $A_{i,j}$ satisfies  the properties of the Lemma.    
$$A_ {  h_{j,i}}\otimes_{T_R(\ell)[a^{-1}]}T_R(\ell)[a^{-1}][b]= M_j(T_R(\ell)[a^{-1}][b])^{ h_{j,i}}.$$
We have for the two centers
$$Z_ {  h_{j,i}}\otimes_{T_R(\ell)[a^{-1}]}T_R(\ell)[a^{-1}][b]=  T_R(\ell)[a^{-1}][b] ^{ h_{j,i}}.$$
Since, by construction,   $T_R(\ell)[a^{-1}][b]$ is \'etale over $T_R(\ell)[a^{-1}]$ the extension  $Z_ {  h_{j,i}}\supset {T_R(\ell)[a^{-1}]}$  is unramified.  

We claim that  $A_ {  h_{j,i}}$ is Azumaya of rank $j^2$ over its center. Let us use  Artin's Theorem, clearly $A_ {  h_{j,i}}$ satisfies all PI's of $j\times j$ matrices, furthermore if $I$ is a proper  ideal of  $A_ {  h_{j,i}}$ we have that $I[b]$ is a proper ideal of $M_j(T_R(\ell)[a^{-1}][b])^{ h_{j,i}}$. Now the quotient of $M_j(T_R(\ell)[a^{-1}][b])^{ h_{j,i}}$ by any ideal is a direct sum of $j\times j$ matrix algebras over commutative rings so  such a quotient does not   satisfy all PI's of $j-1\times j-1$ matrices.
\end{proof}

\begin{theorem}\label{tit}
 
The ideal  
 $J(m_1,\ldots,m_k;a_1,\ldots,a_k)   $, given by Formula \eqref{idueid}, is the ideal of $\mathcal C_n(\ell)$ vanishing on the subvariety $\overline\Sigma (\underline a ;\underline m)$.

\end{theorem} 
\begin{proof}

Since $\ell\geq 2$,  it is enough to restrict  only to $\ell$--tuples which generate  $R:=F(\underline m ;\underline a )$.

 Thus $J(m_1,\ldots,m_k;a_1,\ldots,a_k)$   vanishes on the set of $\ell$--tuples generating a subalgebra  conjugate to $F(\underline m ;\underline a )$ so it vanishes on $\overline\Sigma (\underline a ;\underline m)$.  

By Geometric Invariant Theory  $J(\underline m ;\underline a )$ parametrizes  the semisimple representations of the quotient algebra $\mathcal S_n(\ell)/ I(\underline m ;\underline a )$.  Since   $J(\underline m ;\underline a )$ is a  prime ideal of 
$\mathcal C_n(\ell)$ it is enough to prove  that an open set  of the variety  associated to $J(\underline m ;\underline a )$ is contained  in $\Sigma (\underline a ;\underline m)$. We claim that this follows by Lemma \ref{unle}, by taking the open set $\mathcal O_a$ where the element $a$ is not zero.

Use the notations of Formula \eqref{iD2}.

Then if $\phi: T_R(\ell)[a^{-1}]\to F$  is a point in $\mathcal O_a$ we have, by hypothesis. $ Z_{  h_{j,i}}\otimes_{T_R(\ell)[a^{-1}] ]}F=F^{h_{j,i}}$ and, since $A_ {  h_{j,i}}$ is a rank $j^2$  Azumaya algebra and $F$ is algebraically closed
 \begin{equation}\label{iD30}
\mathcal S_n(\ell)/ I(\underline m ;\underline a )[a^{-1}]\otimes F \simeq \oplus_{i=1}^s  \oplus_j A_ {  h_{j,i}}\otimes_{Z_j} F^{h_{j,i}}
=\oplus_{i=1}^s  \oplus_j M_ {  j}(F)^{h_{j,i}}\end{equation}The trace is given by a Formula as in \ref{Fma} and $p$ is contained  in $\Sigma (\underline a ;\underline m)$.

  This Theorem is related to Theorem 6.5  of \cite{agpr}. 

   \end{proof}
   \begin{corollary}\label{maxx}Let  $M$ be a maximal trace ideal of  the  trace algebra $\mathcal S_n(\ell)/ I(\underline m ;\underline a )$.  Then $\mathcal  S_n(\ell)/ I(\underline m ;\underline a )/M$ is isomorphic, as trace algebra, to an algebra of type  
$F(\underline m' ;\underline a' ) \subset F(\underline m ;\underline a )$. 
\end{corollary}
\begin{proof}
This follows from the previous Theorem and Remark \ref{destra}.
\end{proof}
\section{The  $T$--ideals}

\subsection{The  normalization Theorem}

We next analyze in some detail the  ideals $J(\underline m ;\underline a ), I(\underline m ;\underline a )$.  

Consider  the space $F(\underline m ;\underline a )^\ell$ and inside it the open set  $\mathcal O (\underline a ;\underline m):=\Sigma^\sharp (\underline a ;\underline m)\cap F(\underline m ;\underline a )^\ell$ of generating $\ell$--tuples.

We have that, Lemma \ref{eaai}, two points in $\mathcal O (\underline a ;\underline m)$ are in the same orbit under $Aut(\underline m ;\underline a )$ if and only if they are in the same $PGL(n,F)$ orbit. So we have a factorization of quotient maps,
\begin{equation}\label{foq}
\xymatrix{
F(\underline m ;\underline a )^\ell\ar[r]^i \ar[d]^f
& M_n(F)^\ell\ar[d]^{\bar f}\\   
F(\underline m ;\underline a )^\ell //Aut(\underline m ;\underline a )\ar[r]^{\bar i} & M_n(F)^\ell//PGL(n,F)\ }
\end{equation} such that, by Lemma \ref{int}, the map 
$$\bar i:\mathcal O (\underline m ;\underline a )^\ell //Aut(\underline m ;\underline a )\to \Sigma (\underline a ;\underline m) $$
is bijective and, since $\Sigma (\underline a ;\underline m) $ is smooth, it must be an isomorphism.

As a consequence, setting $A_{(\underline m );\ell}=\bigotimes_{i=1}^k A_{m_i,\ell}$ to be the algebra of polynomial functions on $F(\underline m ;\underline a )^\ell =\oplus_i M_{m_i}(F)$ we have:
\begin{theorem}\label{norm} The map  
$\pi: F(\underline m ;\underline a )^\ell //Aut(\underline m ;\underline a )\to \overline\Sigma (\underline a ;\underline m) $  is the normalization.  The invariant algebra $A_{(\underline m  );\ell}^{Aut(\underline m ;\underline a )}$  is the normalization of the algebra  $\mathcal C_n(\ell)/J(\underline m ;\underline a )$.
\end{theorem}
\begin{proof} By Formula \eqref{datu}
\begin{align}\label{levq}
F(\underline m ;\underline a )^\ell //Aut(\underline m ;\underline a )&=\left(F(\underline m ;\underline a )^\ell //G_0\right)//H \\
 (F(\underline m ;\underline a )^\ell //G_0)&=\prod_{i=1}^k M_{m_i}^\ell//PGL(m_i,F) .
\end{align} This last variety is the variety of $k $--tuples   of semisimple representations each of dimensions $m_i$. As $F(\underline m ;\underline a )^\ell //Aut(\underline m ;\underline a )$ is normal and by the previous remark the map $\pi$  is birational, it only needs to be verified that the map $\pi$ is finite.  For this it is enough to see that the map $$F(\underline m ;\underline a )^\ell // \prod_{i=1}^kPGL(m_i,F)\to F(\underline m ;\underline a )^\ell //Aut(\underline m ;\underline a )\to \overline\Sigma (\underline a ;\underline m) $$ is finite.  This follows from the set theoretic description of the two varieties as semisimple representations since a semisimple representation can be presented as a direct sum only in finitely many ways.
 \end{proof}
 \subsection{Changing the trace}

 We need now a general fact. Take an algebra with trace $R$;  change the trace from $tr$  to $Tr$  by multiplication by $ a\in F,\ a\neq 0,\ Tr(r):=a\cdot tr(r)$.  Set $R_ a$ to be the algebra $R$ with this new trace.  

One has that   trace identities  of $R_ a$  correspond bijectively to    trace identities  of $R  $ by the  isomorphism, of the free algebra with trace,  $\phi_ a$ mapping $\phi_ a:tr(M)\to  a^{-1} \cdot Tr(M)$. In particular if  $R$ is a $ k$ Cayley--Hamilton algebra and $a\in\mathbb N$, then one has that $R_ a$ is an $n= a\cdot k$ Cayley--Hamilton algebra. 

There is still an  isomorphism between the two spaces of trace identities.   A special case is the algebra $F(k;a)$  which is just $M_k(F)$ with trace $t(r)=a\cdot tr(r)$.
\begin{proposition}\label{mula}
The relatively free algebra in $\ell$ variables of  $F(k,a)$ is $S_k(\ell)_a$, i.e. $S_k(\ell)$   with the new trace  $a\cdot tr(r)$.
\end{proposition}
\smallskip

\subsection{The  algebra   of    equivariant maps  }

Let $(\underline m ;\underline a )=(m_1,\ldots,m_k;a_1,\ldots,a_k),\ \sum_ia_im_i=n$. Now  we introduce:
\begin{definition}\label{seee}
Denote by $S_\ell(\underline m ;\underline a )$ the  algebra  of  $Aut(\underline m ;\underline a )$ equivariant maps  from $F(\underline m ;\underline a )^\ell $  to $F(\underline m ;\underline a ) $.  

By $T_\ell(\underline m ;\underline a )$ its trace algebra and $Z_\ell(\underline m ;\underline a )\supset T_\ell(\underline m ;\underline a )$ its center.
\end{definition} Let  $F(\underline m ;\underline a )=\oplus _{i=1}^kM_{m_i}(F) $ recall the notation $A_{\ell,n}$, \pageref{aelln}, for the polynomial functions on $M_n(F)^\ell$.  If we denote by $A_{(\underline m );\ell}$  the algebra of polynomial functions on $F(\underline m ;\underline a )^\ell$, then  $A_{(\underline m );\ell}=\bigotimes_{i=1}^k A_{\ell,m_i} $. We have:
\begin{equation}\label{snml}
S_\ell(\underline m ;\underline a )=(A_{(\underline m );\ell} \otimes F(\underline m ;\underline a ))^{Aut(\underline m ;\underline a )}.
\end{equation}
By Lemma \ref{dueg}, as soon as $\ell\geq 2$ the center $T_\ell(\underline m ;\underline a )  $ of    $S_\ell(\underline m ;\underline a )  $ is
\begin{equation}\label{snml1}
T_\ell(\underline m ;\underline a )=(A_{(\underline m );\ell} \otimes F^k )^{Aut(\underline m ;\underline a )} 
\end{equation} with $F^k$ the center of  $F(\underline m ;\underline a ) $.\smallskip

Since the group   group $Aut(\underline m ;\underline a )$ is the semidirect product $G_0\ltimes H$ with its   conneced component $G_0=\prod_{i=1}^kPGL(m_i,F)$ and $H$ a finte group, first we consider the larger algebra  $S_\ell^0(\underline m ;\underline a )$ of  $G_0$ (Formula  \eqref{datu}) equivariant maps  from $F(\underline m ;\underline a )^\ell $  to $F(\underline m ;\underline a ) $:
$$S_\ell^0(\underline m ;\underline a )=(F(\underline m ;\underline a )\otimes_FA_{(\underline m );\ell})^{G_0  }=\oplus_{j=1}^k (M_{m_i}(F)\otimes \bigotimes_{i=1}^k A_{\ell,m_i})^{\prod_{i=1}^kPGL(m_i,F)  }$$
$$=  \oplus_{j=1}^k (M_{m_j}(F)\otimes A_{\ell,m_j})^{  PGL(m_j,F)  }\otimes \bigotimes_{i=1,i\neq j}^k A_{\ell,m_i} ^{ PGL(m_i,F)  }$$
\begin{equation}\label{pltcc}=\oplus_{i=1}^k (M_{m_i}(F)\otimes A_{  m _i;\ell})^{PGL(m_i) }\otimes_{T(m_i,\ell)}  A_{(\underline m );\ell} ^{G_0  }=\oplus_{i=1}^k S_{  m _i} (\ell) \otimes_{T_{m_i} (\ell) }A_{(\underline m );\ell} ^{G_0  } .\end{equation}

This is the direct sum  of the algebras of $PGL(m_i,F)$ equivariant maps from $M_{m_i}(F)^\ell$ to $M_{m_i}(F)$, that is the usual trace algebras of generic matrices extended to $A_{(\underline m );\ell} ^{G_0  }=T_{m_1} (\ell) \otimes T_{m_2} (\ell) \otimes\ldots\otimes T_{m_k} (\ell) $.   Finally
\begin{equation}\label{dopp}
S_\ell (\underline m ;\underline a )=S_\ell^0(\underline m ;\underline a )^H.
\end{equation}

\subsection{The structure of   equivariant maps  }  So next we want to  analyze the algebra  $S_\ell(\underline m ;\underline a )$.

\begin{lemma}\label{res} The kernel of the restriction    of  the functions    $\mathcal S_n(\ell)$ (resp  $\mathcal C_n(\ell)$)  to  $F(\underline m ;\underline a )^\ell $ is  $I(\underline m ;\underline a )$ (resp  $J(\underline m ;\underline a )$).

The algebra  $\mathcal S_n(\ell)$  (resp  $\mathcal C_n(\ell)$) maps, under the restriction   of  the functions    to  $F(\underline m ;\underline a )^\ell $,  to  the  algebra $S_\ell(\underline m ;\underline a )$   (resp  $T_\ell(\underline m ;\underline a )$).

\end{lemma}
\begin{proof}
By Theorem \ref{tit}
 the ideal $I(m_1,\ldots,m_k;a_1,\ldots,a_k)$ is the ideal of $\mathcal S_n(\ell)$ vanishing on the subvariety $\overline\Sigma (\underline a ;\underline m)$.

Since $PGL(n,F)(F(\underline m ;\underline a )^\ell )$ is dense in $\overline\Sigma (\underline a ;\underline m)$ and the elements of $\mathcal S_n(\ell)$ are 
$PGL(n,F)$ equivariant, then the first statement follows.

As for the second recall that  the elements    $p\in F(\underline m ;\underline a )^\ell $ are the fixed  points of  the invertible elements of the centralizer. Thus,   under a   $PGL(n,F)$ equivariant map,   such a point $p$ is sent to $F(\underline m ;\underline a )$.  Moreover since  $Aut(\underline m ;\underline a )$ is induced by a subgroup of  $PGL(n,F)$, then the second statement also holds. \end{proof}The $k$ indices appearing in  $\underline m ;\underline a$ decompose into $t$ subsets $I_j$ each of some cardinality $u_j$  classes of the equivalence   $(m_i,a_i)=(m_j,a_j)$  and $H=\prod_{j=1}^t S_{u_j}$, of Formula \eqref{datu}, where   $S_{u_j}$ permutes $u_j$ factors  of some type  $M_{m(j),a(j)}(F)^\ell$.

\begin{theorem}\label{ultt} The algebra  $S_\ell(\underline m ;\underline a )$ is isomorphic to
\begin{equation}\label{pltcc1}
S_\ell (\underline m ;\underline a )\simeq\oplus_{j=1}^t \left(\oplus_{i=1}^{u_j}S_{  m (j)} (\ell) \otimes_{T_{ m (j)  } (\ell) }A_{(\underline m );\ell} ^{G_0  }\right)^{S_{u_j}}.
\end{equation}
It contains the algebra, with $R= F(\underline m ;\underline a )$: \begin{equation}\label{pltcc2x}  \mathcal S_n(\ell)/I(\underline m ;\underline a )\otimes _{\mathcal C_n(\ell)/J(\underline m ;\underline a )}A_{(\underline m );\ell}^{Aut(\underline m ;\underline a )}\stackrel{\eqref{notaf}}=F_R(\ell)\otimes_{T_R(\ell)}A_{(\underline m );\ell}^{Aut(\underline m ;\underline a )}.\end{equation}

\end{theorem}
\begin{proof}
We write Formula \eqref{pltcc} as:
\begin{equation}\label{pltcc2}
S_\ell^0(\underline m ;\underline a )= \oplus_{j=1}^t (\oplus_{i=1}^{u_j}S_{  m (j)} (\ell) \otimes_{T_{ m(j )} (\ell) }A_{(\underline m );\ell} ^{G_0  }).\end{equation}

  If the summand   relative to $i\in I_j$ corresponds to  the pair  $(m(j),a(j))$, then the trace in  $S_{  m (j)} (\ell) \otimes_{T_{m(j)} (\ell) }A_{(\underline m );\ell} ^{G_0  }
$ is the ordinary trace multiplied by $a(j)$, by proposition \ref{mula}.  Formula \eqref{pltcc1} follows from  Formula \eqref{pltcc}.  

The group $H=\prod_{j=1}^tS_{u_j}$  acts  on the algebra of Formula  \eqref{pltcc2} by permuting the summands of  each  term $ \oplus_{i=1}^{u_j}S_{  m (j)} (\ell) \otimes_{T_{ m_j } (\ell) }A_{(\underline m );\ell} ^{G_0  }$ through  $S_{u_j}$ hence the claim of   Formula  \eqref{pltcc1}.  The map   $ \mathcal S_n(\ell)\mapsto   \oplus_{i=1}^k S_{  m _i} (\ell) $ given by restricting to $F(\underline m ;\underline a )^\ell $ induces a map to $S_\ell (\underline m ;\underline a )$ with kernel   $I(\underline m ;\underline a )$ by Lemma \ref{res}  and $I(\underline m ;\underline a ) \cap T_{n} (\ell)=J(\underline m ;\underline a )$.
 
\end{proof}
As for
$\left(\oplus_{i=1}^{u_j}S_{  m (j)} (\ell) \otimes_{T_{ m (j)  } (\ell) }A_{(\underline m );\ell} ^{G_0  }\right)^{S_{u_j}}$ we use a  general fact. 

Let $R=A^{\oplus h}=A_1\oplus A_2\oplus\ldots\oplus A_h$  be a direct sum of algebras over $\mathbb Q$ and $G$ a finite group acting on $R$ by automorphisms and permuting the summands transitively.  Let $H$  be the subgroup of $G$ fixing the first summand (and permuting the others). Thus choosing $g_i\in G$  with  $g_i\cdot A_1=A_i,\ g_1=1$ we have that $G=\bigcup_{i=1}^hg_i H$  is the coset decomposition. 
\begin{proposition}
The projection  $\pi_1:R=A^{\oplus h}\to A$ on the first summand induces an isomorphism  between $R^G$ and $A^H$.
\end{proposition}
\begin{proof}
Let $(a_1,a_2,\ldots,a_h)\in R^G$. If $h\in H$, then we have 
$$h\cdot (a_1,a_2,\ldots,a_h)= (h\cdot a_1,a_2,\ldots,a_h)= (a_1,a_2,\ldots,a_h)\implies a_1\in A^H.$$
Next, since $G$   permutes the summands transitively, if $a_1=0$,  then $ a_i=g_i a_ 1 =0,\ \forall i$  so $\pi_1$ is injective.

Finally $\pi_1$ is surjective since, if $a\in A^H$,  then we have $$\frac1{|H|}\sum_{g\in G} g a=\frac1{|H|}\sum_{i=1}^h\sum_{h\in H }g_ih a=(a, g_2 a,g_3a,\ldots,g_ha)\in R^G.$$ 
\end{proof}
Assume that  the $u_j$  indices   which correspond to the pair $m(j),a(j)$  are  $ v+1,\ldots,v+u_j$, then write $$T_{m_1} (\ell) \otimes T_{m_2} (\ell) \otimes \ldots\otimes T_{m_k} (\ell)= B\otimes T_{v+1} (\ell) \otimes T_{v+2} (\ell) \otimes \ldots\otimes T_{v+u_j} (\ell)\otimes C $$ and $S_{u_j} $ permutes the factors $T_{v+1} (\ell) \otimes T_{v+2} (\ell) \otimes \ldots\otimes T_{v+u_j} (\ell)$ while $S_{u_j-1} $ permutes the factors $  T_{v+2} (\ell) \otimes \ldots\otimes T_{v+u_j} (\ell)$.
\begin{corollary}\label{strf}
$$\left(\oplus_{i=1}^{u_j}S_{  m (j)} (\ell) \otimes_{T_{ m (j)  } (\ell) }A_{(\underline m );\ell} ^{G_0  }\right)^{S_{u_j}}$$\begin{equation}\label{assu}
\simeq  S_{  m (j)} (\ell) \otimes_{F }  B\otimes (T_{v+2} (\ell) \otimes \ldots\otimes T_{v+u_j} (\ell)) ^{S_{u_j-1}}\otimes C.
\end{equation}
\end{corollary}
In particular this last algebra is a domain and we have a better understanding of Theorem \ref{prst}. We leave to the reader to verify
\begin{equation}\label{iD0l}
 \mathcal S_n(\ell)/ I(\underline m ;\underline a )\subset S_\ell (\underline m ;\underline a ) \subset F_R(\ell)\otimes_{T_R(\ell)}G_R(\ell)\simeq  \oplus_{i=1}^s  \oplus_j D_ {  h_{j,i}}
\end{equation} \medskip

The algebra of Formula \eqref{pltcc1} may be viewed as a form of {\em integral closure} of the algebra of Formula \eqref{pltcc2}. Compare with Theorem \ref{inte}.

We finish with a conjecture.

{\bf Conjecture 2\label{conj1}}  In the universal map of  $\mathrm j: \mathcal S_n(\ell)\to M_n(A)$ into $n\times n$  matrices, the algebra $A$ is the coordinate ring of the closure of the stratum  $ \Sigma^\sharp (\underline a ;\underline m)$.

This will be discussed elsewhere.
   \bigskip
 
 \centerline{REFERENCES}\bigskip
 
\bibliographystyle{amsalpha}

\end{document}